\documentclass[a4paper,11pt]{amsart}

\usepackage[all]{xy}

\usepackage[utf8]{inputenc}		
\usepackage[T1]{fontenc}
\usepackage[english]{babel}

\usepackage{amsfonts}			
\usepackage{amsmath}
\usepackage{amssymb}
\usepackage{amsthm}
\usepackage{mathtools}
\makeatletter					
\@namedef{subjclassname@2020}{	
	\textup{2020} Mathematics Subject Classification}
\makeatother

\usepackage{graphicx}			
\usepackage{tikz}
	\usetikzlibrary{cd}

\usepackage{hyperref}			
	\hypersetup{colorlinks=false, urlcolor=black, linkcolor=black}

\usepackage{enumitem}			

\usepackage{color}				

\newcommand{\N}{\mathbb{N}}						
\newcommand{\R}{\mathbb{R}}						
\newcommand{\C}{\mathbb{C}}						
\newcommand{\K}{\mathbb{K}}						

\renewcommand{\S}{\mathbb{S}}					
\newcommand{\T}{\mathbb{T}}						

\newcommand{\eps}{\varepsilon}					
\newcommand{\derham}{H_{\text{dR}}}				

\newcommand{\dd}								
	{\mathop{}\!\mathrm{d}}						
\newcommand{\ddn}[1]							
	{\mathop{}\!\mathrm{d^{#1}}}

\newcommand{\abs}[1]							
	{\left| #1 \right|}
\newcommand{\smallabs}[1]						
	{\lvert #1 \rvert}	
\newcommand{\norm}[1]							
	{\left\lVert #1 \right\rVert}	
\newcommand{\smallnorm}[1]						
	{\lVert #1 \rVert}						
\newcommand{\ip}[2]								
	{\left< #1 , #2 \right>}
\newcommand{\smallip}[2]
	{\langle #1 , #2 \rangle}
\newcommand{\dfpart}[1]
	{\left< #1 \right>}

\DeclareMathOperator{\vol}{vol}					
\DeclareMathOperator{\spt}{spt}					
\DeclareMathOperator*{\esssup}{ess\,sup}		

\DeclareMathOperator{\Span}{span}

\DeclareMathOperator{\Cl}{Cl}					
\DeclareMathOperator{\connsum}{\#}

\DeclareMathOperator{\hodge}{\mathtt{\star}}

\let\Re\relax									

\DeclareMathOperator{\Re}{Re}
\DeclareMathOperator{\im}{im}					

\newcommand{\cesob}{W^\text{CE}}
\newcommand{\cesobs}{W^{\text{CE}\sharp}}
\newcommand{\cesobds}{W^{\text{DS}}}

\newcommand{\cehom}[1]{H_{\text{CE}}^{#1}}
\newcommand{\cehoms}[1]{H_{\text{CE}\sharp}^{#1}}
\newcommand{\cehomds}[1]{H_{\text{DS}}^{#1}}

\newcommand{\pharm}[2]{\mathcal{H}_{#1}^{#2}}

\newcommand{\pharmgeq}[2]{\mathcal{H}_{#1}^{#2, \geq}}

\newcommand{\loc}{\mathrm{loc}}

\newcommand{\cA}{\mathcal{A}}
\newcommand{\cB}{\mathcal{B}}
\newcommand{\cE}{\mathcal{E}}
\newcommand{\cI}{\mathcal{I}}
\newcommand{\cJ}{\mathcal{J}}
\newcommand{\cG}{\mathcal{G}}
\newcommand{\cH}{\mathcal{H}}

\newenvironment{acknowledgments}
	{\bigskip\noindent{\bf Acknowledgments.}}{}

\newtheorem{thm}{Theorem}[section]{\bf}{\it}
\newtheorem{lemma}[thm]{Lemma}
\newtheorem{prop}[thm]{Proposition}
\newtheorem{cor}[thm]{Corollary}

{\bf}{\it}

{\bf}{\it}

\theoremstyle{definition}
\newtheorem{defn}[thm]{Definition}

\theoremstyle{remark}
\newtheorem{rem}[thm]{Remark}

\numberwithin{equation}{section}

\begin{document}

\title[Conformally formal manifolds and UQR non-ellipticity]{Conformally formal manifolds and the uniformly quasiregular non-ellipticity of $(\mathbb{S}^2 \times \mathbb{S}^2) \operatorname{\#} (\mathbb{S}^2 \times \mathbb{S}^2)$}
\author{Ilmari Kangasniemi}
\address{University of Helsinki, Department of Mathematics and Statistics, P.O. Box 68 (Pietari Kalmin Katu 5), FI-00014, Finland, and Syracuse University, Department of Mathematics, Syracuse, NY 13244, USA}
\email{kikangas 'at' syr.edu}

\thanks{This work was supported by the doctoral program DOMAST of the University of Helsinki, the Academy of Finland project \#297258, and the Simons Semester \emph{Geometry and analysis in function and mapping theory on Euclidean and metric measure spaces} at IMPAN, Warsaw.}

\begin{abstract}
	We show that the manifold $(\S^2 \times \S^2) \connsum (\S^2 \times \S^2)$  does not admit a non-constant non-injective uniformly quasiregular self-map. This answers a question of Martin, Mayer, and Peltonen, and provides the first example of a quasiregularly elliptic manifold which is not uniformly quasiregularly elliptic.
	
	To obtain the result, we introduce conformally formal manifolds, which are closed smooth $n$-manifolds $M$ admitting a measurable conformal structure $[g]$ for which the $(n/k)$-harmonic $k$-forms of the structure $[g]$ form an algebra. This is a conformal counterpart to the existing study of geometrically formal manifolds. We show that, similarly as in the geometrically formal theory, the real cohomology ring $H^*(M; \R)$ of a conformally formal $n$-manifold $M$ admits an embedding of algebras $\Phi \colon H^*(M; \R) \hookrightarrow \wedge^* \R^n$. We also show that uniformly quasiregularly elliptic manifolds $M$ are conformally formal in a stronger sense, in which the wedge product is replaced with a conformally scaled Clifford product. For this stronger version of conformal formality, the image of $\Phi$ is closed under the Euclidean Clifford product of $\wedge^* \R^n$, which in turn is impossible for $M = (\S^2 \times \S^2) \connsum (\S^2 \times \S^2)$.
\end{abstract}

\maketitle

\section{Introduction}

\emph{Quasiconformal} and \emph{quasiregular} maps are a geometric generalization of holomorphic maps to higher dimensions. In particular, given two oriented Riemannian $n$-manifolds $M$ and $N$, a map $f \colon M \to N$ is \emph{$K$-quasiregular} for $K \geq 1$ if $f$ is continuous, belongs to the local Sobolev space $W^{1,n}_\loc(M, N)$, and satisfies the infinitesimal distortion condition
\[
	\abs{Df(x)}^n \leq K J_f(x)
\]
for almost every $x \in M$. Here, $\abs{\cdot}$ stands for the operator norm, and $J_f$ for the Jacobian determinant. A $K$-quasiregular homeomorphism is then called \emph{$K$-quasiconformal}. If $n=2$, it is well known that the set of $1$-quasiregular maps consists exactly of holomorphic maps between Riemann surfaces.

Moreover, a self-map $f \colon M \to M$ on an oriented Riemannian $n$-manifold is \emph{uniformly $K$-quasiregular} if every iterate $f^j$ of $f$ is $K$-quasiregular. Uniformly quasiregular maps are thus quasiregular self-maps which behave well under iteration, and they therefore provide a geometric higher dimensional generalization of holomorphic dynamics.

Let $M$ be a closed, connected, oriented Riemannian $n$-manifold. Then $M$ is \emph{quasiregularly elliptic} if there exists a non-constant quasiregular map $f \colon \R^n \to M$. The study of such manifolds traces back to questions of Gromov \cite[p.\ 67]{Gromov_book} and Rickman \cite[p.\ 183]{Rickman_QREllipt_question}. Similarly, we call $M$ \emph{uniformly quasiregularly elliptic} if there exists a non-constant non-injective uniformly quasiregular self-map $f \colon M \to M$. The question of characterizing uniformly quasiregularly elliptic manifolds is described by Martin, Mayer, and Peltonen \cite{Martin-Mayer-Peltonen} as a non-injective version of a conjecture by Lichnerowicz \cite{Lichnerowicz_conjecture} solved by Lelong-Ferrand \cite{Lelong-Ferrand_LichneowiczSolution}.

The two concepts are related in that if $M$ is uniformly quasiregularly elliptic, then it is also quasiregularly elliptic; see \cite[Theorem 1.1]{Martin-Mayer-Peltonen} or \cite[Theorem 5.7]{Kangaslampi_thesis}. Whether there exists a converse implication or not has however so far remained unsolved. When $n = 2$, it follows from the uniformization theorem that quasiregular and uniformly quasiregular ellipticity are equivalent. Moreover, it was shown by Kangaslampi \cite[Theorem 7.1]{Kangaslampi_thesis} that quasiregular and uniformly quasiregular ellipticity remain equivalent when $n = 3$.

As our main result, we show that for $n = 4$,  quasiregular and uniformly quasiregular ellipticity are in fact not equivalent. In particular, let $M$ be the connected sum $(\S^2 \times \S^2) \connsum (\S^2 \times \S^2)$ of two copies of $\S^2 \times \S^2$. It was shown by Rickman \cite{Rickman_ConnSumQREllipt} that $M$ is quasiregularly elliptic. Due to this, it was asked by Martin, Mayer, and Peltonen in \cite[p.\ 2093]{Martin-Mayer-Peltonen} whether $M$ is uniformly quasiregularly elliptic; see also \cite[p.\ 338]{Astola-Kangaslampi-Peltonen}, and \cite[p.\ 1442]{Martin-ICM2010Survey}. We resolve this question to the negative.

\begin{thm}\label{thm:uqr_ellipticity_is_stronger}
	The manifold $(\S^2 \times \S^2) \connsum (\S^2 \times \S^2)$ is not uniformly quasiregularly elliptic.
\end{thm}

In a slightly more general form, the concrete topological obstruction for uniformly quasiregular ellipticity which results in Theorem \ref{thm:uqr_ellipticity_is_stronger} is as follows.

\begin{thm}\label{thm:4_manifold_uqr_obstruction}
	Suppose that $M$ is a closed, connected, oriented, and uniformly quasiregularly elliptic 4-manifold. Let $b_2 = b_2^+ + b_2^-$ be the standard decomposition of the second Betti number of $M$ into positive definite and negative definite parts under the intersection form. Then $b_2^+, b_2^- \in \{0, 1, 3\}$.
\end{thm}

In particular, for $M = (\S^2 \times \S^2) \connsum (\S^2 \times \S^2)$ we have $b_2^+ = b_2^- = 2$, and therefore the uniformly quasiregular ellipticity of $M$ is ruled out by Theorem \ref{thm:4_manifold_uqr_obstruction}. We note that each of the individual values of $b_2^+$ and $b_2^-$ permitted by Theorem \ref{thm:4_manifold_uqr_obstruction} is assumed by a manifold known to be uniformly quasiregularly elliptic. Indeed, for $\S^4$ we have $b_2^+ = b_2^- = 0$, for $\S^2 \times \S^2$ we have $b_2^+ = b_2^- = 1$, and for $\T^4$ we have $b_2^+ = b_2^- = 3$. For a proof of the uniformly quasiregular ellipticity of these manifolds, see e.g.\ \cite{Astola-Kangaslampi-Peltonen}.

\subsection{Conformally formal manifolds}

In our approach towards Theorems \ref{thm:uqr_ellipticity_is_stronger} and \ref{thm:4_manifold_uqr_obstruction}, we develop a significant amount of more general obstruction theory, which is of independent interest. We now proceed to present the main points of this theory.

Let $M$ be a closed smooth $n$-manifold. A smooth Riemannian metric $g$ on $M$ is called \emph{formal} if its space of harmonic forms is closed under the wedge product, and therefore forms an algebra. Recall that a smooth $k$-form $\omega \in C^\infty(\wedge^k M)$ is \emph{harmonic} if $(d^* d + d d^*)\omega = 0$, where the dependence of harmonic forms on the choice of metric $g$ is through the codifferential operator $d^*$. Hodge theory yields that if $M$ is closed, then $\omega \in C^\infty(\wedge^k M)$ is harmonic if and only if it satisfies the pair of partial differential equations
\begin{align*}
	d\omega &= 0,\\
	d^* \omega &= 0.
\end{align*}

The $n$-manifold $M$ is then called \emph{geometrically formal} if it admits a smooth formal Riemannian metric. This definition is due to Kotschick \cite{Kotschick_duke}, following the ideas of Sullivan \cite{Sullivan_formal_metrics}. Geometrically formal manifolds and their topological obstructions have been studied in e.g.\ \cite{Kotschick_duke},  \cite{Grojsean-Nagy_GeometricallyFormal}, \cite{Bar_FormalMetricsWithPosCurvature} and \cite{Kotschick_FormalNonnegScalarCurvature}.

In the theory leading up to our main results, we end up considering a conformal version of formal Riemannian metrics. Namely, suppose that $[g]$ is a bounded measurable conformal structure on $M$. Then for every index $k \in \{1, \dots, n-1\}$, the structure $[g]$ defines a conformally invariant space $\pharm{g}{k}(M; \R)$ of \emph{$(n/k)$-harmonic $k$-forms}. In particular, elements $\omega \in \pharm{g}{k}(M; \R)$ are weak solutions of the pair of conformally invariant differential equations
\begin{gather}
	\label{eq:diffeq_d} d\omega = 0,\\
	\label{eq:diffeq_dstar} d \abs{\omega}_g^{\frac{n-k}{k} - 1} \hodge_g \omega = 0,
\end{gather}
where $g$ is any measurable Riemannian metric contained in the structure $[g]$, and $\hodge_g$ is the Hodge star with respect to $g$. The elements of $\pharm{g}{k}(M; \R)$ are also studied under the more general class of \emph{$\cA$-harmonic forms}, where in this case $\cA$ is a non-linear operator depending on the conformal structure $[g]$.

The definition of $\pharm{g}{k}(M; \R)$ extends to $k = 0$ by having $\pharm{g}{0}(M; \R)$ be the set of constant functions, which are exactly the weak solutions of \eqref{eq:diffeq_d} among 0-forms. The extension to $k = n$ is more subtle. The condition \eqref{eq:diffeq_dstar} for an $n$-form $\omega$ is essentially a condition for $\hodge_g \omega$ to not change sign. We interpret this condition by defining a conformally invariant subspace $\pharmgeq{g}{n}(M; \R)$ of the space of measurable $n$-forms by
\begin{align}
	\pharmgeq{g}{n}(M; \R) &= \left\{ \omega \text{ meas.\ } n \text{-form} :
	\hodge_g \omega \geq 0, \text{ or }
	\hodge_g \omega \leq 0 \right\}\label{eq:pharmgeq},
\end{align}
where the inequalities are understood to hold almost everywhere.

With these definitions in place, we show the following.

\begin{thm}\label{thm:uqr_elliptic_is_formal}
	Let $M$ be a closed, connected, oriented Riemannian $n$-mani\-fold for $n \geq 2$. Suppose that $M$ is uniformly quasiregularly elliptic. Then there exists a bounded conformal structure $[g]$ on $M$ which satisfies the following property: there exists a space $\pharm{g}{n}(M; \R) \subset \pharmgeq{g}{n}(M; \R)$ such that $\pharm{g}{*}(M; \R) = \bigoplus_{j=0}^n \pharm{g}{j}(M; \R)$ is a graded $\R$-algebra with respect to $+$ and $\wedge$.
\end{thm}

Due to how the above condition for $[g]$ resembles the definition of a formal Riemannian metric, we call such a structure $[g]$ \emph{conformally formal}. Similarly, a closed, oriented, smooth Riemannian $n$-manifold $M$ admitting such a structure $[g]$ is also called \emph{conformally formal}. This definition makes no reference to a uniformly quasiregular map, and can therefore be studied independently of the uniformly quasiregularly elliptic setting. 

In a certain sense the definition of conformally formal structures appears even more demanding than that of formal Riemannian metrics. Not only does the definition include that $\pharm{g}{*}(M; \R)$ is closed under the wedge product, but it also requires that the spaces $\pharm{g}{k}(M; \R)$ are linear despite the non-linearity of the differential equation \eqref{eq:diffeq_dstar}. Regardless, the move from the smooth setting to the measurable setting and the challenges surrounding the case $k = n$ interfere with the use of many arguments from the geometrically formal theory.

However, some of the methodology does survive the change in setting. We recall that a conjecture of Bonk and Heino\-nen \cite[p.\ 222]{Bonk-Heinonen_Acta} on quasiregularly elliptic manifolds was recently given a positive solution by Prywes \cite{Prywes_Annals}. The conjecture predicted a sharp upper bound of $2^n$ on the dimension of the real cohomology ring of a quasiregularly elliptic $n$-manifold. Besided Prywes' proof, an alternate proof of the uniformly quasiregularly elliptic special case was also given by the author in \cite{Kangasniemi_CohomBound}. Most notably, the same bound also holds in the setting of geometrically formal manifolds by \cite[Theorem 6]{Kotschick_duke}, and the proofs for uniformly quasiregularly elliptic manifolds and for geometrically formal manifolds bear some similarity.

In particular, the proof of the main cohomological obstruction of \cite{Kangasniemi_CohomBound} transfers entirely to the setting of conformally formal manifolds. Moreover, we in fact obtain a stronger version of the result by taking advantage of the wedge product properties of conformally formal structures.

\begin{thm}\label{thm:conf_formal_algebra_embedding}
	Let $M$ be a closed, connected, oriented, smooth $n$-manifold. Suppose that $M$ is conformally formal. Then there exists an embedding of graded algebras $\Phi \colon H^*(M; \R) \to \wedge^* \R^n$ which maps the cup product to the exterior product. Moreover, the image of $\Phi$ is closed under the Hodge star.
\end{thm}

Theorem \ref{thm:conf_formal_algebra_embedding} already reveals some new manifolds which are not uniformly quasiregularly elliptic. For an example of this, let $M = \connsum^{15} (\S^2 \times \S^4)$. This manifold $M$ still satisfies the cohomological dimension restriction proven in \cite{Kangasniemi_CohomBound} and \cite{Prywes_Annals}, but $H^*(M; \R)$ can not be realized as a subalgebra of $\wedge^* \R^6$. Indeed, any such realization has to map $H^2(M; \R)$ surjectively onto $\wedge^2 \R^6$, but the cup product of any two elements of $H^2(M; \R)$ is zero.


\subsection{Clifford algebra}

Theorem \ref{thm:4_manifold_uqr_obstruction} follows from a refined version of the previously discussed obstruction theory, obtained by an application of Clifford algebras. We note that, to our knowledge, Clifford algebras have not seen significant prior use in the study of higher dimensional quasiconformal analysis. In our case, we use a measurable Riemannian metric $g$ to obtain a Clifford product $\cdot_g$ on $\wedge^* T^*_x M$ for almost every $x \in M$. Consequently, we may multiply elements of $\pharm{g}{*}(M; \R)$ with each other under this Clifford product.

However, the Clifford product $\cdot_g$ is dependent on the exact choice of metric $g$ in a conformal structure $[g]$. To eliminate this dependence, we instead study a scaled version of the Clifford product. Namely, suppose that $k, l, m \in \{1, \dots, n\}$, and let $\omega_1$ and $\omega_2$ be a measurable $l$-form and $m$-form, respectively. We use $\dfpart{\omega}_k$ to denote the $(\wedge^k T^*M)$-component of a differential form $\omega \colon M \to \wedge^* T^* M$, and define an operation $\odot_g$ by
\[
	\dfpart{\omega_1 \odot_g \omega_2}_k 
	= \abs{\dfpart{\omega_1 \cdot_g \omega_2}_k}_g^{\frac{k}{l+m} - 1} \dfpart{\omega_1 \cdot_g \omega_2}_k.
\]
This operation $\odot_g$ is non-linear, but in exchange depends only on the conformal structure $[g]$. Moreover, $\dfpart{\omega_1 \odot_g \omega_2}_k = \omega_1 \wedge \omega_2$ if $k = l + m$. A more precise definition of $\odot_g$ is given in Section \ref{sect:Clifford_alg}.

We then proceed to improve Theorem \ref{thm:uqr_elliptic_is_formal} in the following way.

\begin{thm}\label{thm:uqr_elliptic_is_clifford}
	Let $M$ be a closed, connected, oriented Riemannian $n$-mani\-fold for $n \geq 2$. Suppose that $M$ is uniformly quasiregularly elliptic. Then there exists a bounded conformal structure $[g]$ on $M$ which satisfies the following property: there exists a space $\pharm{g}{n}(M; \R) \subset \pharmgeq{g}{n}(M; \R)$ such that $\pharm{g}{*}(M; \R)$ is closed under $+$ and $\odot_g$.
\end{thm}

Since structures $[g]$ such as in Theorem \ref{thm:uqr_elliptic_is_clifford} are also conformally formal, we call such structures \emph{conformally formal in the Clifford sense}. With the additional information provided by this property, the embedding of Theorem \ref{thm:conf_formal_algebra_embedding} improves in the following way, finally yielding the general result behind the topological obstruction of Theorem \ref{thm:4_manifold_uqr_obstruction}.

\begin{thm}\label{thm:conf_formal_clifford_embedding}
	Let $M$ be a closed, connected, oriented, smooth $n$-manifold. Suppose that $M$ is conformally formal in the Clifford sense. Then there exists an embedding of graded algebras $\Phi \colon H^*(M; \R) \to \wedge^* \R^n$ which maps the cup product to the exterior product. Moreover, the image of $\Phi$ is closed under the Euclidean Clifford product of $\wedge^* \R^n$.
\end{thm}

\subsection{A remark on complex coefficients}

In our definition of conformally formal manifolds, we used differential forms with real coefficients. However, as was also the case in \cite{Kangasniemi_CohomBound}, our methods require the use of differential forms with complex coefficients. Due to this, we end up having to also define \emph{conformally $\C$-formal} bounded conformal structures $[g]$, where the spaces $\pharm{g}{k}(M; \R)$ in the definition are replaced by the corresponding spaces $\pharm{g}{k}(M; \C)$ of differential forms with complex coefficients. Conformal formality in the Clifford sense is also similarly extended to complex coefficients. 

The complex versions of conformal formality are easily seen to imply the corresponding real versions, but unlike in the case of harmonic forms, the non-linearity of \eqref{eq:diffeq_dstar} makes it unclear to us whether the converse holds. We prove Theorems \ref{thm:uqr_elliptic_is_formal}, \ref{thm:conf_formal_algebra_embedding}, \ref{thm:uqr_elliptic_is_clifford} and \ref{thm:conf_formal_clifford_embedding} for both real and complex coefficients, with complex versions of the statements given with the proofs.

\subsection{Structure of this paper}

Our discussion splits into essentially three parts. The first part, consisting of Sections \ref{sect:p_harm_recap}, \ref{sect:conf_formality} and \ref{sect:Clifford_alg}, focuses on the algebraic theory of conformally formal structures. In it, we recall the necessary prerequisites of measurable conformal structures, $p$-harmonic forms and the Clifford product. Moreover, we define conformally formal structures in detail, and prove some of the initial algebraic consequences of the definition.

The second part, consisting of Sections \ref{sect:Hodge_theory}, \ref{sect:conf_cohomology} and \ref{sect:consequences}, focuses on topological obstructions to conformal formality. We first recall some prerequisites from non-linear Hodge theory. We then discuss conformal cohomology, and use it to prove Theorems \ref{thm:conf_formal_algebra_embedding} and \ref{thm:conf_formal_clifford_embedding}. Finally, we discuss obstructions implied by Theorems \ref{thm:conf_formal_algebra_embedding} and \ref{thm:conf_formal_clifford_embedding}, including the obstruction given in Theorem \ref{thm:4_manifold_uqr_obstruction}.

The third part, consisting of Sections \ref{sect:qr_and_structs} and \ref{sect:formality_of_uqr_ell}, is the quasiregular part. In it, we recall the necessary prerequisites of quasiregular maps, and then prove Theorems \ref{thm:uqr_elliptic_is_formal} and \ref{thm:uqr_elliptic_is_clifford}. 

\begin{acknowledgments}
	The author thanks Pekka Pankka for introducing him to the subject, as well as for continued mathematical discussions and guidance. Moreover, the initial ideas which started this work were conceived during a visit to IMPAN for the Simons Semester in Fall 2019, and the author thanks the semester organizers and participants for creating an inspiring research environment.
\end{acknowledgments}


\section{Conformal structures and $p$-harmonic forms}\label{sect:p_harm_recap}

In this section, we recall the necessary prerequisites of bounded conformal structures, and of $p$-harmonic forms with real or complex coefficients.

\subsection{Bounded conformal structures}

Let $M$ be a closed, connected, oriented, smooth $n$-manifold. A \emph{conformal structure} on $M$ is an equivalence class $[g]$ of measurable Riemannian metrics under the equivalence relation of multiplication with a positive measurable function. In particular, two measurable Riemannian metrics $g$ and $g'$ on $M$ belong to the same conformal structure if $g' = \rho^2 g$ a.e. on $M$ for some positive measurable function $\rho \colon M \to (0, \infty)$.

Given two conformal structures $[g]$ and $[g']$ on $M$, their \emph{conformal distance} is defined by
\[
	d([g], [g']) = \esssup_{x \in M} \left( \max \left\{ 
		\log \left( \frac{\abs{v}_g}{\abs{w}_g} \right) :
		v, w \in T_x M, \abs{v}_{g'} = \abs{w}_{g'} = 1 
	\right\}\right).
\]
We then say that a conformal structure $[g]$ is \emph{bounded} if there exists a smooth Riemannian metric $g_0$ on $M$ such that $d([g], [g_0]) < \infty$. Note that, due to the compactness of $M$, we have $d([g_0], [g_0']) < \infty$ for all smooth Riemannian metrics $g_0, g_0'$ on $M$. Hence, for a bounded conformal structure $[g]$ on a closed manifold $M$, we in fact have $d([g], [g_0]) < \infty$ for all smooth Riemannian metrics $g_0$ on $M$.

\subsection{Differential forms of real and complex coefficients}

Suppose then that $[g]$ is a bounded conformal structure on a closed, connected, oriented, smooth $n$-manifold $M$. Let $g_0$ be a smooth Riemannian metric on $M$, the choice of which makes $M$ into a Riemannian manifold. Throughout this paper, we use $\K$ to denote either $\R$ or $\C$, when the discussion is applicable to both fields of coefficients.

Given $k \in \{0, \dots, n\}$, we denote by $\Gamma(\wedge^k M; \R)$ the space of \emph{measurable differential $k$-forms} on $M$. Moreover, we let $\Gamma(\wedge^k M; \C) = \Gamma(\wedge^k M; \R) \otimes \C$ be the space of \emph{measurable differential $k$-forms with complex coefficients}. In particular, elements $\omega \in \Gamma(\wedge^k M; \C)$ are pairs $\omega = \alpha + i\beta$, where $\alpha, \beta \in \Gamma(\wedge^k M; \C)$. 

A measurable Riemannian metric $g$ defines a.e.\ on $M$ a point-wise inner product $\ip{\cdot}{\cdot}_g$ and a point-wise norm $\abs{\cdot}_g$ for elements of $\Gamma(\wedge^k M; \K)$. Given $\omega \in \Gamma(\wedge^k M; \K)$ and $p \in [1, \infty]$, we denote by $\norm{\omega}_{p, g}$ the $L^p$-norm of $\omega$ with respect to $g$, that is,
\[
	\norm{\omega}_{p, g} = \left( \int_M \abs{\omega}_g^p \vol_g\right)^{\frac{1}{p}}.
\] 
The space of $\omega \in \Gamma(\wedge^k M; \K)$ with $\norm{\omega}_{p, g_0} < \infty$ is denoted $L^p(\wedge^k M; \K)$. 

Although the specific norms $\norm{\omega}_{p, g_0}$ vary depending on $g_0$, the spaces $L^p(\wedge^k M; \K)$ are independent on the choice of smooth $g_0$ due to the compactness of $M$. Hence, we may discuss the spaces $L^p(\wedge^k M; \K)$ on a closed, oriented, smooth Riemannian $n$-manifold, even without fixing a smooth metric $g_0$. Moreover, for the space with complex coefficients, we in fact have $L^p(\wedge^k M; \C) = L^p(\wedge^k M; \R) \otimes \C$, i.e.\ a complex differential form is $L^p$-integrable if and only if its real and imaginary part are.

We note that there exists a unique $g \in [g]$ for which $\vol_{g} = \vol_{g_0}$. This special element in a conformal structure is particularly useful when carrying out computations, since for this metric $g$ we have
\begin{equation}\label{eq:norm_comparison}
C^{-1} \abs{\cdot}_{g_0} \leq \abs{\cdot}_{g} \leq C \abs{\cdot}_{g_0}
\end{equation}
a.e.\ on $M$. Here, $C$ is only dependent on $n$ and the conformal distance $d([g], [g_0])$. In particular, for this choice of $g$, the norms $\norm{\cdot}_{p, g}$ and $\norm{\cdot}_{p, g_0}$ are comparable.

We also note that for a specific exponent, the norm $\norm{\omega}_{p, g}$ is in fact independent on the choice of $g \in [g]$. Namely, if $g' = \rho^2 g$ and $\omega \in \Gamma(\wedge^k M; \K)$, then $\abs{\omega}_{g'} = \rho^{-k} \abs{\omega}_g$ and $\vol_{g'} = \rho^{n} \vol_g$. Hence,
\[
	\norm{\omega}_{\frac{n}{k}, g'} = \norm{\omega}_{\frac{n}{k}, g},
\]
where for $k = 0$ we interpret $n/k = \infty$. This exponent $p = n/k$ is hence called the \emph{conformal exponent}.

\subsection{$p$-harmonic forms of the conformal exponent} 

Suppose that $\omega \in L^1(\wedge^k M; \K)$, where $k < n$. We say that $d\omega \in L^1(\wedge^{k+1} M; \K)$ is a \emph{weak differential} of $\omega$ if, for every smooth real $(n-k-1)$-form $\eta \in C^\infty(\wedge^{n-k-1} M; \R)$, the $n$-forms $\omega \wedge d\eta$ and $d\omega \wedge \eta$ are integrable and satisfy
\[
	\int_M \omega \wedge d\eta = (-1)^{k+1} \int_M d\omega \wedge \eta.
\]

Let $g' \in [g]$. The \emph{Hodge star} $\hodge_{g'}$ with respect to $g'$ is defined for a.e.\ $x \in M$ by the relation
\[
	\alpha \wedge \hodge_{g'} \beta = \ip{\alpha}{\beta}_{g'} \vol_{g'}
\]
for all $\alpha, \beta \in (\wedge^k T^*_x M) \otimes \K$, where $k \in \{0, \dots, n\}$. Now, for $k \in \{1, \dots, n-1\}$ and $p \in (1, \infty)$, we call a $k$-form $\omega \in \Gamma(\wedge^k M; \K)$ \emph{$p$-harmonic} with respect to $g'$, if $\norm{\omega}_{p, g'} < \infty$ and $\omega$ satisfies the weak partial differential equations
\begin{gather*}
	d\omega = 0, \text{ and}\\
	d \left(\abs{\omega}_{g'}^{p-2} \hodge_{g'} \omega\right) = 0.
\end{gather*}
Note that, if $\omega_x = 0$ for $x \in M$, then we interpret that $\abs{\omega_x}_{g'}^{p-2} \hodge_{g'} \omega_x = 0$.

In general, $p$-harmonic forms may depend on the choice of $g'$ from the conformal structure. However, in the conformal exponent $p = n/k$, the dependence is only on the conformal structure $[g]$. We recall the simple proof of this fact.

\begin{lemma}\label{lem:conf_exp_invariance}
	Let $M$ be a closed, connected, oriented smooth $n$-manifold. Let $[g]$ be a bounded conformal structure on $M$, let $g_1, g_2 \in [g]$, and let $\K \in \{\C, \R\}$. Suppose that $\omega \in \Gamma(\wedge^k M; \K)$ is $(n/k)$-harmonic with respect to $g_1$, where $k \in \{1, \dots, n-1\}$. Then $\omega$ is $(n/k)$-harmonic with respect to $g_2$.
\end{lemma} 
\begin{proof}
	As pointed out before, we have $\norm{\omega}_{n/k, g_1} = \norm{\omega}_{n/k, g_2}$. The equation $d\omega = 0$ is dependent only on the smooth structure of $M$. Moreover, since $g_1$ and $g_2$ are conformally equivalent, there exists a $\rho \colon M \to (0, \infty)$ such that $g_2 = \rho^2 g_1$, and hence
	\[
		\abs{\omega}_{g_2}^{\frac{n}{k}-2} \hodge_{g_2} \omega
		= \left(\rho^{-k}\abs{\omega}_{g_1}\right)^{\frac{n}{k}-2}
			\rho^{n-2k} \hodge_{g_1} \omega
		= \abs{\omega}_{g_1}^{\frac{n}{k}-2} \hodge_{g_1} \omega.
	\]
	Therefore, $d (\abs{\omega}_{g_2}^{(n/k)-2} \hodge_{g_2} \omega) = d (\abs{\omega}_{g_1}^{(n/k)-2} \hodge_{g_1} \omega) = 0$.
\end{proof}

As stated in the introduction, for $k \in \{1, \dots, n-1\}$, we denote by $\pharm{g}{k}(M; \K)$ the space of $(n/k)$-harmonic $k$-forms on $M$ with respect to the bounded conformal structure $[g]$. By Lemma \ref{lem:conf_exp_invariance}, the space is independent on choice of representative metric $g$ from within the conformal structure. Moreover, we let $\pharm{g}{0}(M; \K)$ be the space of essentially bounded $0$-forms $\omega$ which solve the weak partial differential equation $d\omega = 0$. Since $M$ is connected, this is exactly the space of almost everywhere constant $\K$-valued functions on $M$.

We recall the fundamental invariance properties of the spaces $\pharm{g}{k}(M; \K)$.

\begin{lemma}\label{lem:conformal_hodge}
	Let $M$ be a closed, connected, oriented smooth $n$-manifold, let $[g]$ be a bounded conformal structure on $M$, and let $\K \in \{\C, \R\}$. Let $\omega \in \pharm{g}{k}(M; \K)$ for $k \in \{1, \dots, n-1\}$. Then the following conditions hold.
	\begin{enumerate}
		\item \label{enum:pharm_prop_scaling} For every $\lambda \in \K$, we have $\lambda \omega \in \pharm{g}{k}(M; \K)$.
		\item \label{enum:pharm_prop_confhodge} We have
		\[
			\abs{\omega}_{g}^{\frac{n}{k}-2} \hodge_{g} \omega \in \pharm{g}{n-k}(M),
		\]
		where the above is interpreted as $0$ at points $x \in M$ where $\omega_x = 0$.
		\item \label{enum:pharm_prop_conj} If $\K = \C$, then we have $\overline{\omega} \in \pharm{g}{k}(M; \C)$, where $\overline{\omega}$ is the \emph{conjugate} of $\omega$ defined by $\overline{\alpha + i \beta} = \alpha - i\beta$ for $\alpha, \beta \in \Gamma(\wedge^k M; \R)$.
	\end{enumerate}
\end{lemma}
\begin{proof}
	Property \eqref{enum:pharm_prop_scaling} follows immediately from the formulas $d(\lambda \omega) = \lambda d\omega$, $\abs{\lambda \omega}_g = \abs{\lambda} \abs{\omega}_g$ and $\hodge_g(\lambda \omega) = \overline{\lambda} \hodge_g \omega$. Similarly, property \eqref{enum:pharm_prop_conj} immediately follows from the formulas $d\overline{\omega} = \overline{d\omega}$, $\abs{\overline{\omega}}_g = \abs{\omega}_g$, and $\hodge_g\overline{\omega} = \overline{\hodge_g \omega}$.
	
	For property \eqref{enum:pharm_prop_confhodge}, we denote $\omega' = \abs{\omega}_{g}^{(n/k)-2} \hodge_{g} \omega$. By \eqref{eq:diffeq_dstar}, we have $d\omega' = 0$. Moreover, we have
	\begin{multline*}
		\abs{\omega'}_{g}^{\frac{n}{n-k}-2} \hodge_{g} \omega'
		= \Bigl( \abs{\omega}_g^{\frac{n-k}{k}}\Bigr)^{\frac{k}{n-k}-1}
			\hodge_g \Bigl(\abs{\omega}_{g}^{\frac{n-k}{k}-1} \hodge_g \omega\Bigr)\\
		= \abs{\omega}_g^{1 - \frac{n-k}{k}} \abs{\omega}_{g}^{\frac{n-k}{k}-1}
			\hodge_g \hodge_g \omega
		= (-1)^{k(n-k)}\omega.
	\end{multline*}
	Hence, by \eqref{eq:diffeq_d} we have $d \abs{\omega'}_{g}^{n/(n-k) - 2} \hodge_g \omega' = 0$. For the final requirement, we have
	\[
		\norm{\abs{\omega}_{g}^{^\frac{n}{k}-2} \hodge_{g} \omega}_{g, \frac{n}{n-k}} 
		= \norm{\omega}_{g, \frac{n}{k}}^\frac{k}{n-k} < \infty. 
	\]
\end{proof}

Finally, we note that $\pharm{g}{k}(M; \R)$ is a subset of $\pharm{g}{k}(M; \C)$ by the identification $\omega = \omega + i0$. More precisely, $\pharm{g}{k}(M; \R)$ consists of exactly the elements of $\pharm{g}{k}(M; \C)$ which have zero imaginary part. Indeed, this is since the $L^p$-norm and the equations \eqref{eq:diffeq_d}-\eqref{eq:diffeq_dstar} are the same for real forms and for  complex forms with no imaginary part. 

\section{Conformal formality}\label{sect:conf_formality}

In this section, we begin the discussion on the conformally formal structures defined in the introduction, and prove several basic algebraic properties of the space $\pharm{g}{*}(M; \K)$ for such structures $[g]$.

\subsection{Conformally formal structures}

We recall the definition of the space $\pharmgeq{g}{n}(M; \R)$ from the introduction: $\pharmgeq{g}{n}(M; \R)$ consists of all measurable $n$-forms $\omega$ for which $\hodge_{g} \omega$ is either non-negative a.e.\ or non-positive a.e. This space is independent of the choice of metric in a conformal structure, since a conformal change of metric multiplies the Hodge star by a positive function.

\begin{rem}
	The space $\pharmgeq{g}{n}(M; \R)$ is an interpretation of the space of solutions for the second equation \eqref{eq:diffeq_dstar} of $(n/k)$-harmonic $k$-forms in the case $k = n$. Indeed, in this case the equation for an $n$-form $\omega$ becomes
	\begin{equation}\label{eq:diffeq_dstar_case_n}
		d \left( \frac{\hodge_g \omega}{\abs{\hodge_g \omega}} \right) = 0.
	\end{equation}
	The quantity being differentiated therefore essentially corresponds to the sign of $\hodge_g \omega$, and \eqref{eq:diffeq_dstar_case_n} requires that this sign is constant. However, \eqref{eq:diffeq_dstar_case_n} becomes ill-posed at points $x \in M$ where $\omega_x = 0$. Hence, obtaining a concrete space of solutions for \eqref{eq:diffeq_dstar_case_n} requires a choice of how to treat such values of $\omega$.
\end{rem}

We now define the complex counterpart $\pharmgeq{g}{n}(M; \C)$. The definition is most easily stated in terms of the real version:
\begin{align*}
	\pharmgeq{g}{n}(M; \C) &= \left\{ \lambda \omega : \lambda \in \C, \omega \in \pharmgeq{g}{n}(M; \R) \right\}.
\end{align*}
The resulting space $\pharmgeq{g}{n}(M; \C)$ essentially consists of $n$-forms $\omega$ for which $\arg (\hodge_g \omega)$ is constant, or in other words, for which the functional $\hodge_g \omega \colon M \to \C$ points in the direction of a single unit complex number at every point of $M$. Similar to the real counterpart, the space $\pharmgeq{g}{n}(M; \C)$ is also conformally invariant.

We now state the definition of conformally formal bounded conformal structures, where the real case is merely recalling the definition from the introduction.

\begin{defn}\label{def:conformally_formal_R_and_C}
	Let $M$ be a closed, connected, oriented, smooth $n$-manifold, let $[g]$ be a bounded conformal structure on $M$, and let $\K \in \{\R, \C\}$. We say that $[g]$ is \emph{conformally $\K$-formal} if there exists a set $\pharm{g}{n}(M; \K) \subset \pharmgeq{g}{n}(M; \K)$ for which $(\pharm{g}{*}(M; \K), +, \wedge)$ is a $\K$-algebra. In case $\K = \R$, we may also omit the $\R$ and merely use the term \emph{conformally formal}.
\end{defn}

\subsection{Algebraic properties}

We begin by pointing out that all vector spaces contained in $\pharmgeq{g}{n}(M; \K)$ are small.

\begin{lemma}\label{lem:one_sign_vector_spaces}
	Let $M$ be a closed, connected, oriented, smooth $n$-manifold, and let $[g]$ be a bounded conformal structure on $M$. Let $V \subset \pharmgeq{g}{n}(M; \K)$ be a vector space with coefficients in $\K$. Then $\dim_\K (V) \leq 1$.
\end{lemma}
\begin{proof}
	Suppose that $\omega_1, \omega_2 \in V \setminus \{0\}$ We wish to show that $\omega_1$ and $\omega_2$ are linearly dependent. We may write $\omega_i = h_i \vol_g$, in which case $\hodge_g \omega_i = h_i$. Since $\omega_i \in \pharmgeq{g}{n}(M; \K)$, by multiplying $\omega_i$ with a unit scalar in $\K$, we may assume that $h_i(x) \in \R$ and $h_i(x) \geq 0$ for a.e.\ $x \in M$.
	
	Let $S_1 = \{ \lambda \in \R : h_1 + \lambda h_2 \geq 0 \text{ a.e.} \}$ and $S_2 = \{ \lambda \in \R : h_1 + \lambda h_2 \leq 0 \text{ a.e.} \}$. Since
	\[
	 \omega_1 + \lambda \omega_2 
	 \in V \cap \Gamma(\wedge^n M; \R) 
	 \subset \pharmgeq{g}{n}(M; \K) \cap \Gamma(\wedge^n M; \R)
	 = \pharmgeq{g}{n}(M; \R)
	\] 
	for every $\lambda \in \R$, we have $S_1 \cup S_2 = \R$. The sets $S_1$ and $S_2$ are closed in $\R$. Moreover, both of them are nonempty, since $\lambda \in S_1$ for all non-negative $\lambda$, and $\lambda \in S_2$ for all sufficiently small negative $\lambda$. Hence, there exists a $\lambda \in S_1 \cap S_2$. But now $\omega_1 + \lambda \omega_2 = 0$ where $\lambda \in \R \subset \K$, which shows that $\omega_1$ and $\omega_2$ are linearly dependent. The claim follows.
\end{proof}

We now show that, for conformally formal structures $[g]$, the algebra $\pharm{g}{*}(M; \K)$ satisfies the key property which was used in the proof of the main results of \cite{Kangasniemi_CohomBound}. Its counterpart in the theory of formal manifolds is \cite[Lemma 4]{Kotschick_duke}, which states that for a formal Riemannian metric, the inner product of harmonic forms is constant. The conformally formal counterpart is not quite as strong, but still implies a significant amount of rigidity provided by our assumption.

\begin{lemma}\label{lem:abs_value_rigidity}
	Let $M$ be a closed, connected, oriented, smooth $n$-manifold, let $[g]$ be a bounded conformal structure on $M$, and let $\K \in \{\R, \C\}$. Suppose that $[g]$ is conformally $\K$-formal. Then for every measurable Riemannian metric $g \in [g]$, and every choice of $\pharm{g}{n}(M; \K)$ such that $\pharm{g}{*}(M; \K)$ is an algebra, there exists a non-negative measurable real function $\rho_g \colon M \to [0, \infty)$ with the following property: given any $\omega_1, \omega_2 \in \pharm{g}{k}(M; \K)$ for $k > 0$, we have
	\[
		\ip{\omega_1}{\omega_2}_g = A \rho_g^{2k}
	\]
	a.e.\ on $M$ for some constant $A \in \K$. In particular, for any $\omega \in \pharm{g}{k}(M; \K)$ for $k > 0$, we have
	\[
		\abs{\omega}_g = B \rho_g^k
	\]
	a.e.\ on $M$ for some non-negative real constant $B \in [0, \infty)$.
\end{lemma}
\begin{proof}
	Fix a $g \in [g]$ and a selection of $\pharm{g}{n}(M; \K)$. By Lemma \ref{lem:one_sign_vector_spaces}, we have $\dim \pharm{g}{n}(M; \K) \in \{0, 1\}$. In either case, there exists a measurable function $\rho_g \colon M \to [0, \infty)$ such that every element of $\eta \in \pharm{g}{n}(M; \K)$ is of the form
	\[
		\eta = C \rho_g^n \vol_g
	\]
	for some $C \in \K$. Indeed, if $\dim \pharm{g}{n}(M; \K) = 0$, this holds by default for any function $\rho$ since the only element in $\pharm{g}{n}(M; \K)$ is the zero function. And if $\dim \pharm{g}{n}(M; \K) = 1$, we may select any $\eta \in \pharm{g}{n}(M; \K) \setminus \{0\}$ and define $\rho_g = (\lambda (\hodge_g \eta))^{1/n}$, where $\lambda$ is the unique unit scalar in $\K$ for which $\lambda (\hodge_g \eta_x) \in \R$ and $\lambda (\hodge_g \eta_x) \geq 0$ for a.e.\ $x \in M$.
	
	In particular, our claim now holds for $\rho_g$ in the case $k = n$. Suppose now that $k \in \{1, \dots, n-1\}$. Let $\omega_1 \in \pharm{g}{k}(M; \K)$. By Lemma \ref{lem:conformal_hodge}, we have $\omega_1' = \abs{\omega_1}_g^{(n/k) - 2} \hodge_g \omega_1 \in \pharm{g}{n-k}(M; \K)$. Since $\pharm{g}{*}(M; \K)$ is closed under the wedge product, we hence have $\omega_1 \wedge \omega_1' \in \pharm{g}{n}(M; \K)$. But $\omega_1 \wedge \omega_1' = \abs{\omega_1}_g^{n/k} \vol_g$, and also by the selection criterion of $\rho_g$ we have $\omega_1 \wedge \omega_1' = C_1 \rho_g^n \vol_g$ for some $C_1 \in \K$. Hence,
	\[
		\abs{\omega_1}_g^{\frac{n}{k}} = C_1 \rho_g^n
	\]
	a.e.\ on $M$. Since $\abs{\omega_1}_g$ and $\rho_g$ are non-negative and real, we must have $C_1 \in \R$ and $C_1 \geq 0$, and therefore
	\[
		\abs{\omega_1}_g = C_1^\frac{k}{n} \rho_g^k
	\]
	a.e.\ on $M$.
	
	Finally, let also $\omega_2 \in \pharm{g}{k}(M; \K)$. Now we must have $\omega_2 \wedge \omega_1' \in \pharm{g}{n}(M; \K)$, and therefore $\omega_2 \wedge \omega_1' = C_2 \rho_g^n \vol_g$ for some $C_2 \in \K$. But $\omega_2 \wedge \omega_1' = \abs{\omega_1}_g^{(n/k) - 2}\ip{\omega_1}{\omega_2}_g \vol_g $ , so we obtain
	\[
		\abs{\omega_1}_g^{\frac{n}{k} - 2}\ip{\omega_1}{\omega_2}_g
		= C_2 \rho_g^n
	\]
	a.e.\ on $M$. By the previous result on $\abs{\omega_1}_g$, we obtain
	\[
		\left( C_1^\frac{k}{n} \rho_g^k \right)^{\frac{n}{k}-2}
			\ip{\omega_1}{\omega_2}_g
		= C_2 \rho_g^n
	\]
	a.e.\ on $M$, which when rearranged yields
	\[
		\ip{\omega_1}{\omega_2}_g = C_1^{\frac{2k}{n} - 1} C_2 \rho_g^{2k}
	\]
	a.e.\ on $M$. This concludes the proof.
\end{proof}

As an immediate consequence, we obtain a result that the elements of $\pharm{g}{k}(M; \K)$ for a conformally formal $g$ share the same support. We also point out here the similarity of this to the ideas behind the second main result of \cite{Kangasniemi_CohomBound}. 

\begin{cor}\label{cor:same_support}
	Let $M$ be a closed, connected, oriented, smooth $n$-manifold, and let $[g]$ be a conformally $\K$-formal bounded conformal structure on $M$, where $\K \in \{\R, \C\}$. Then all elements of $\pharm{g}{k}(M; \K) \setminus \{0\}$ for all $k \in \{1, \dots, n\}$ share the same support.	
\end{cor}
\begin{proof}
	For every $\omega \in \pharm{g}{k}(M; \K) \setminus \{0\}$ with $k \in \{1, \dots, n\}$, we have $\spt \omega = \spt \rho_g$, where $\rho_g$ is provided by Lemma \ref{lem:abs_value_rigidity}. 
\end{proof}

Finally, we conclude the uniqueness of $\pharm{g}{n}(M; \K)$, which follows from methods similar to those used in the proof of Lemma \ref{lem:abs_value_rigidity}.

\begin{lemma}\label{lem:Hn_representation}
	Let $M$ be a closed, connected, oriented, smooth $n$-manifold, let $[g]$ be a bounded conformal structure on $M$, and let $\K \in \{\R, \C\}$. Suppose that $[g]$ is conformally $\K$-formal, and moreover, that $\pharm{g}{k}(M; \K) \neq \{0\}$ for some $k \in \{1, \dots, n-1\}$. Then the set $\pharm{g}{n}(M; \K)$ is unique, consisting exactly of elements of the form $C \abs{\omega}_g^{n/k} \vol_g$ for any $\omega \in \pharm{g}{k}(M; \K) \setminus \{0\}$, where $C \in \K$.
\end{lemma}
\begin{proof}
	By our assumption, there exists a $\omega \in \pharm{g}{k}(M; \K) \setminus \{0\}$. By Lemma \ref{lem:conformal_hodge}, we have $\omega' = \abs{\omega}_{g}^{(n/k)-2} \hodge_{g} \omega \in \pharm{g}{n-k}(M; \K) \setminus \{0\}$. Hence, we obtain that $\abs{\omega}_g^{n/k} \vol_g = \omega \wedge \omega' \in \pharm{g}{n}(M; \K)$. Since $\abs{\omega}_g^{n/k} \vol_g$ is nonzero, and since $\pharm{g}{n}(M; \K)$ is at most one-dimensional due to Lemma \ref{lem:one_sign_vector_spaces}, we have $\pharm{g}{n}(M; \K) = \{C \abs{\omega}_g^{n/k} \vol_g : C \in \K \}$ and the claim follows.
\end{proof}

\section{Clifford algebras}\label{sect:Clifford_alg}

In this section, we recall the necessary prerequisite information on the Clifford product, with the main focus on how it extends the wedge product of differential forms. We then provide a more precise definition of conformal structures which are conformally formal in the Clifford sense. For a more comprehensive discussion on Clifford algebras, we refer the reader to one of the many available introductory texts, such as \cite{Garling_CliffordAlgebras} or \cite{Lounesto_CliffordAlgebras}. 

\subsection{Clifford products}

Let $V$ be a finite-dimensional vector space over $\K$. Moreover, let $b \colon V \times V \to \K$ be a bilinear form on $V$. The \emph{Clifford algebra} $\Cl(V, b)$ is the free associative unital $\K$-algebra generated by $V$ subject to the relation $v \cdot v = b(v, v)$ for $v \in V$. More formally, $\Cl(V, Q)$ is the quotient of the tensor algebra $\otimes^* V$ by the ideal generated by the set $\{ v \otimes v - b(v, v) 1 : v \in V \}$. The multiplication operation $\cdot$ of $\Cl(V, b)$ is called the \emph{Clifford product}.

In this paper, we restrict our attention to Clifford algebras induced by inner products. Suppose that $V$ is an $n$-dimensional inner product space over $\K$. If $\K = \R$, then the inner product $\ip{\cdot}{\cdot}$ on $V$ is a bilinear form, and therefore induces a Clifford algebra $\Cl(V, \ip{\cdot}{\cdot})$. The situation for $\K = \C$ is slightly more involved, since a complex inner product is conjugate linear in its second coordinate. Hence, in order to obtain a Clifford algebra $\Cl(V, \ip{\cdot}{\cdot})$ for $\K = \C$, we have to assume further structure on $V$. 

In particular, in the case $\K = \C$, we also suppose that we have selected a conjugation map $v \mapsto \overline{v}$ on $V$, which is a conjugate-linear self-map on $V$ satisfying $\overline{\overline{v}} = v$ for $v \in V$. Then the complex inner product $\ip{\cdot}{\cdot}$ yields a bilinear form $b$ on $V$ by $b(v, w) = \ip{v}{\overline{w}}$ for $v, w \in V$. Therefore, we may define $\Cl(V, \ip{\cdot}{\cdot}) = \Cl(V, b)$. We note that selecting a conjugation map on $V$ is equivalent to selecting a subspace of real elements $W \subset V$ for which $V = W \otimes \C$. In particular, the conjugation map is obtained from $W$ by $\overline{w \otimes z} = w \otimes \overline{z}$ for $w \in W, z \in \C$, while $W$ is obtained from the conjugation map as the space of self-conjugate elements of $V$.

As a vector space, the Clifford algebra $\Cl(V, \ip{\cdot}{\cdot})$ is canonically isomorphic to $\wedge^* V$. Through this isomorphism, we may consider the Clifford product $\cdot$ as an alternate product on $\wedge^* V$ different from the wedge product $\wedge$. The products $\cdot$ and $\wedge$ are moreover related in the following way: if we denote by $\dfpart{w}_k$ the $\wedge^k V$-component of an element $w \in \wedge^* V$, then for all $v_1 \in \wedge^{k_1} V$ and $v_2 \in \wedge^{k_2} V$ we have
\begin{equation}\label{eq:wedge_clifford_relation}
	v_1 \wedge v_2 = \dfpart{v_1 \cdot v_2}_{k_1 + k_2}.
\end{equation}

\begin{rem}
	For the sake of the reader less familiar with Clifford algebras, we briefly detail the identification $\Cl(V, \ip{\cdot}{\cdot}) \cong \wedge^* V$ through the lens of an induced basis.
	
	Namely, let $v, w \in V$. The space $V$ is naturally contained in $\Cl(V, \ip{\cdot}{\cdot})$. Moreover, the defining relation $v \cdot v = \ip{v}{\overline{v}}$ for elements of $V$ can be rewritten in the form
	\begin{equation}\label{eq:clifford_mult_relation}
		v \cdot w + w \cdot v = \ip{v}{\overline{w}} + \ip{w}{\overline{v}}.
	\end{equation}
	Indeed, this follows from the relations $(v+w) \cdot (v+w) = \ip{v+w}{\overline{v+w}}$, $v \cdot v = \ip{v}{\overline{v}}$, and $w \cdot w = \ip{w}{\overline{w}}$, along with the use of the corresponding distributive laws for the inner and Clifford products.

	Suppose then that $B = \{e_1, \dots, e_n\}$ is an orthonormal basis of real elements of $V$. Here, we call an element $v \in V$ real if $\overline{v} = v$. Let $i, j \in \{1, \dots, n\}$ and $i \neq j$. Then we have $\ip{e_i}{\overline{e_i}} = \ip{e_i}{e_i} = 1$. Consequently,
	\begin{equation}\label{eq:clifford_basis_power}
		e_i \cdot e_i = 1.
	\end{equation}
	Moreover, by \eqref{eq:clifford_mult_relation} we have $e_i \cdot e_j + e_j \cdot e_i = \ip{e_i}{\overline{e_j}} + \ip{e_j}{\overline{e_i}} = \ip{e_i}{e_j} + \ip{e_j}{e_i} = 0$. Therefore,
	\begin{equation}\label{eq:clifford_basis_anticommute}
		e_i \cdot e_j = - e_j \cdot e_i.
	\end{equation}
	
	By the two rules \eqref{eq:clifford_basis_power} and \eqref{eq:clifford_basis_anticommute}, any arbitrarily large product of elements of $B$ can be reduced into the form $\pm e_{i_1} \cdots e_{i_k}$, where $i_1 < \dots < i_k$. It then follows from the bilinearity of the Clifford product that any arbitrarily large product of elements of $V$ is a linear combination of elements $e_{i_1} \cdots e_{i_k}$ as above. In fact, such elements form a basis of $\Cl(V, \ip{\cdot}{\cdot})$, which we call the \emph{induced basis} of $B$; for details, see e.g.\ \cite[Sections 5.1-5.2]{Garling_CliffordAlgebras}.
	
	We may therefore identify $\Cl(V, \ip{\cdot}{\cdot})$ with $\wedge^* V$ as a vector space by identifying $e_{i_1} \cdots e_{i_k}$ with $e_{i_1} \wedge \dots \wedge e_{i_k}$ for $i_1 < \dots < i_k$. The relation \eqref{eq:wedge_clifford_relation} between the wedge and Clifford products becomes clear from this identification. Indeed, if two basis elements $e_{i_1} \cdots e_{i_k}$ and $e_{j_1} \cdots e_{j_l}$ have no common factors $e_i$, then the anti-commutativity rule \eqref{eq:clifford_basis_anticommute} yields that their wedge and Clifford products remain the same. On the other hand, if a common factor $e_i$ exists, then the wedge product vanishes, and the Clifford product is of lower order than $k+l$ due to \eqref{eq:clifford_basis_power}.
	
	The above construction of the identification $\Cl(V, \ip{\cdot}{\cdot}) \cong \wedge^* V$ depends on the choice of basis $B$. However, every choice of orthonormal real basis $B$ in fact results in the same underlying map, and hence the identification is canonical. For a coordinate-free approach to this isomorphism, see e.g.\ the proof of \cite[Theorem 5.2.1]{Garling_CliffordAlgebras}.
\end{rem}

Due to its importance for our applications, we point out the following simple norm estimate for the Clifford product. We provide a short proof for the convenience of the reader.

\begin{lemma}\label{lem:Clifford_holder}
	Let $(V, \ip{\cdot}{\cdot})$ be an $n$-dimensional inner product space over $\K$, where $\K \in \{\R, \C\}$. If $\K = \C$, then suppose moreover that $V$ has a fixed conjugation map $v \mapsto \overline{v}$. Let $\cdot$ be the Clifford product induced by $\ip{\cdot}{\cdot}$ on $\wedge^* V$, and let $\abs{\cdot}$ be the norm induced by $\ip{\cdot}{\cdot}$ on $\wedge^j V$ for all $j \in \{0, \dots, n\}$. 
	
	Then there exists a $C = C(n) > 0$ for which the following holds: for all $k, k_1, k_2 \in \{0, \dots, n\}$, and for all $v_i \in \wedge^{k_i} V$ for $i = 1, 2$, we have 
	\[
		\abs{\dfpart{v_1 \cdot v_2}_k} \leq C \abs{v_1} \abs{v_2}.
	\]
\end{lemma}
\begin{proof}
	Let $\{e_i : i = 1, \dots, n\}$ be an orthonormal basis of self-conjugate elements of $V$, and denote by $e_I$ and $e_J$ the induced basis elements on $\wedge^{k_1} V$ and $\wedge^{k_2} V$, respectively, where $I \in \cI$ and $J \in \cJ$. Note that for all $I$ and $J$, $\dfpart{e_I \cdot e_J}_k$ is either a basis element, opposite of a basis element, or zero. We write $v_1 = \sum_{I \in \cI} a_I e_I$ and $v_2 =  \sum_{J \in \cJ} b_J e_J$. Then
	\begin{multline*}
		\abs{\dfpart{v_1 \cdot v_2}_k}
		= \abs{\sum_{I \in \cI, J \in \cJ} a_I b_J \dfpart{e_I \cdot e_J}_k}
		\leq \sum_{I \in \cI, J \in \cJ} \abs{a_I} \abs{b_J}\\
		= \left( \sum_{I \in \cI} \abs{a_I}\right)
			\left( \sum_{J \in \cJ} \abs{b_J}\right)
		\leq \sqrt{\binom{n}{k_1} \binom{n}{k_2}} \abs{v_1} \abs{v_2}
		\leq 2^n \abs{v_1} \abs{v_2}.
	\end{multline*}
\end{proof}

\subsection{Clifford products and conformal structures}

Suppose then that $M$ is a closed, connected, oriented smooth $n$-manifold. Let $g$ be a measurable Riemannian metric on $M$. Then for almost every $x \in M$, the inner product $\ip{\cdot}{\cdot}_g$ on $T^*_x M \otimes \K$ induces a Clifford product on $\wedge^* T^*_x M \otimes \K$, which we denote by $\cdot_g$ to emphasize the dependence on $g$. Consequently, if $\omega_1, \omega_2 \in \Gamma(\wedge^* M; \K)$, then we obtain a measurable $\omega_1 \cdot_g \omega_2 \in \Gamma(\wedge^* M; \K)$. We note that if $\K = \C$, then the real subspace $T^*_x M \subset T_x^* M \otimes \C$ provides us with the conjugation map used to define $\cdot_g$. 

Consider now a metric $\rho^2 g$ which is conformally equivalent to $g$. Then the conformal change of metric affects the Clifford product in the following way: if $\alpha \in \Gamma(\wedge^l M; \K)$ and $\beta \in \Gamma(\wedge^m M; \K)$, then
\begin{equation}\label{eq:clifford_conformal}
	\dfpart{\alpha \cdot_{\rho^2 g} \beta}_k 
	= \rho^{k - l - m} \dfpart{\alpha \cdot_{g} \beta}_k.
\end{equation}
Indeed, this easily follows from \eqref{eq:clifford_basis_power} and \eqref{eq:clifford_basis_anticommute} along with the fact that, if $\{\eps_i\}$ is an $\ip{\cdot}{\cdot}_{g}$-orthonormal basis of self-conjugate elements of $T^*_x M \otimes \K$, then $\{\rho \eps_i\}$ is an $\ip{\cdot}{\cdot}_{\rho^2g}$-orthonormal basis of self-conjugate elements of $T^*_x M \otimes \K$.

Hence, the Clifford product $\cdot_g$ depends on the choice of $g \in [g]$. However, as discussed in the introduction, this dependence may be eliminated by introducing a non-linear version of the product with a suitable scaling term, which we denote by $\odot_g$. Namely, suppose that $\alpha \in \Gamma(\wedge^l M; \K)$ and $\beta \in \Gamma(\wedge^m M; \K)$. If $l = m = 0$, we define $\alpha \odot_g \beta = \alpha \beta$ using the usual product on $\K$. Otherwise, we define $\alpha \odot_g \beta$ by
\begin{equation}\label{eq:scaled_clifford_def}
	\alpha \odot_g \beta = 1 + \sum_{k=1}^n 
		\frac{\dfpart{\alpha \cdot_g \beta}_k}
			{\abs{\dfpart{\alpha \cdot_g \beta}_k}_g^{\frac{l+m-k}{l+m}}}.
\end{equation}
Note that, in case $\dfpart{\alpha \cdot_g \beta}_k = 0$, we interpret the resulting $0/0$-expression in \eqref{eq:scaled_clifford_def} as $0$, since the power in the denominator is smaller than 1. The exception to this would be if $k = 0$, and for this reason we have directly defined the result for $k = 0$ to always be 1 in order to avoid issues. We then extend $\odot_g$ to $\Gamma(\wedge^* M; \K) \times \Gamma(\wedge^* M; \K)$ by summing over the results for all components $\Gamma(\wedge^l M; \K) \times \Gamma(\wedge^m M; \K)$. 

We observe that, if $\alpha \in \Gamma(\wedge^l M; \K)$ and $\beta \in \Gamma(\wedge^m M; \K)$ with $m \neq 0$ or $l \neq 0$, then
\begin{multline*}
	\alpha \odot_{\rho^2 g} \beta 
	= 1 + \sum_{k=1}^n \frac{\rho^{k-l-m}\dfpart{\alpha \cdot_g \beta}_k}
	{\left(\rho^{-k}\abs{\rho^{k-l-m}\dfpart{\alpha \cdot_g \beta}_k}_g
		\right)^{\frac{l+m-k}{l+m}}}\\
	= 1 + \sum_{k=1}^n \frac{\rho^{k-l-m}\dfpart{\alpha \cdot_g \beta}_k}
	{\rho^{k-l-m}\abs{\dfpart{\alpha \cdot_g \beta}_k}_g^{\frac{l+m-k}{l+m}}}
	= \alpha \odot_{g} \beta.
\end{multline*}
Moreover, if $\alpha \in \Gamma(\wedge^0 M; \K)$ and $\beta \in \Gamma(\wedge^0 M; \K)$, then $\alpha \odot_{\rho^2 g} \beta = \alpha \beta = \alpha \odot_g \beta$.  Hence, the operation $\odot_g$ depends only on the conformal structure $[g]$, and not on the choice of metric $g$ within $[g]$.

With this, we state the full definition of conformal formality in the Clifford sense, including both the cases $\K = \R$ and $\K = \C$.

\begin{defn}\label{def:conformally_clifford_formal}
	Let $M$ be a closed, connected, oriented, smooth $n$-manifold, let $[g]$ be a bounded conformal structure on $M$, and let $\K \in \{\R, \C\}$. We say that $[g]$, as well as the manifold $M$ admitting $[g]$, is \emph{conformally $\K$-formal in the Clifford sense}, if there exists a set $\pharm{g}{n}(M; \K) \subset \pharmgeq{g}{n}(M; \K)$ for which $\pharm{g}{*}(M; \K)$ is closed under both addition and the scaled Clifford product $\odot_g$. In case $\K = \R$, we may again also omit the $\R$ from the term. 
\end{defn}

If $\alpha \in \Gamma(\wedge^l M; \K)$ and $\beta \in \Gamma(\wedge^m M; \K)$, then $\dfpart{\alpha \odot_g \beta}_{l+m} = \alpha \wedge \beta$. Hence, if $[g]$ is conformally $\K$-formal in the Clifford sense, it is also conformally $\K$-formal in the sense of Section \ref{sect:conf_formality}. Thus, the entire theory discussed in Section \ref{sect:conf_formality} also applies to structures which are conformally formal in the Clifford sense.

\begin{rem}
	Note that the scaling used in \eqref{eq:scaled_clifford_def} is not the only possible option. Indeed, one alternative approach would have been to sum over the terms
	\[
	\abs{\alpha}^{(k-l-m)/(2l)}\abs{\beta}^{(k-l-m)/(2m)}\dfpart{\alpha \cdot_g \beta}_k,
	\]
	and many other similar conformally invariant approaches can also be devised. However, for the most part these different approaches to scaling yield the same structures $[g]$ which are conformally $\K$-formal in the Clifford sense, essentially due to Lemma \ref{lem:abs_value_rigidity}.
\end{rem}

We end this section by showing that, instead of verifying that $\pharm{g}{*}(M; \K)$ is closed under $\odot_g$, it suffices to verify that a basis is mapped inside $\pharm{g}{*}(M; \K)$ by $\odot_g$. In the case of the wedge product this is immediately clear due to bilinearity, but since $\odot_g$ is non-linear, an extra argument is required.

\begin{lemma}\label{lem:clifford_nonlinear_basis_lemma}
	Let $M$ be a closed, connected, oriented, smooth $n$-manifold, let $[g]$ be a bounded conformal structure on $M$, and let $\K \in \{\R, \C\}$. Suppose that there exists a vector space $\pharm{g}{n}(M; \K) \subset \pharmgeq{g}{n}(M; \K)$ for which $\pharm{g}{*}(M; \K)$ is closed under addition. Moreover, suppose that there exists a finite graded \linebreak basis $\{\omega_i \in \pharm{g}{k}(M; \K) : i \in I, 0 \leq k \leq n\}$ of $\pharm{g}{*}(M; \K)$ such that $\omega_i \odot_g \omega_j \in \pharm{g}{*}(M; \K)$ for all $i, j \in I$. Then $[g]$ is conformally formal in the Clifford sense.
\end{lemma}
\begin{proof}
	Let $i, j \in I$, and let $\omega_i \in \pharm{g}{l}(M; \K)$ and $\omega_j \in \pharm{g}{m}(M; \K)$. We then have $\omega_i \wedge \omega_j = \dfpart{\omega_i \odot_g \omega_j}_{l+m} \in \pharm{g}{*}(M; \K)$. Since the wedge product is bilinear, it follows that $\pharm{g}{*}(M; \K)$ is closed under the wedge product, and therefore $[g]$ is conformally $\K$-formal in the non-Clifford sense.
	
	Let then $\alpha \in \pharm{g}{l}(M; \K)$, $\beta \in \pharm{g}{m}(M; \K)$, and let $k \in \{0, \dots, n\}$. It suffices to show that $\dfpart{\alpha \odot_g \beta}_k \in \pharm{g}{k}(M; \K)$. Since elements of $\pharm{g}{0}(M; \K)$ are constant and $\pharm{g}{*}(M; \K)$ is closed under scalar multiplication, we may assume that $(l, m) \neq (0, 0)$. Moreover, since $\dfpart{\alpha \odot_g \beta}_0$ is the constant function 1 by definition, we may also assume $k > 0$.
	
	If $\omega_i \in \pharm{g}{l}(M; \K)$ and $\omega_j \in \pharm{g}{m}(M; \K)$ are basis elements for $i, j \in I$, then by our assumption we have $\omega_{ij} = \dfpart{\omega_i \odot_g \omega_j}_k \in \pharm{g}{k}(M; \C)$. Moreover, from the definition of $\odot_g$ we see that $\abs{\omega_{ij}}_g = \smallabs{\langle \omega_1 \cdot_{g} \omega_2 \rangle_k}^{k/(l+m)}_g$. Hence, multiplying the equation $\dfpart{\omega_1 \odot_{g} \omega_2}_k = \omega_{ij}$ on both sides by $\abs{\omega_{ij}}^{(l+m)/k - 1}_g$, we obtain
	\[
		\dfpart{\omega_1 \cdot_{g} \omega_2}_k 
		= \abs{\omega_{ij}}_g^{\frac{l+m}{k} - 1} \omega_{ij}.
	\]
	
	Consequently, we may use the bilinearity of the Clifford product, and write
	\[
		\dfpart{\alpha \cdot_{g} \beta}_k 
		= \sum_{i, j \in I} a_{ij} \abs{\omega_{ij}}_g^{\frac{l+m}{k} - 1} \omega_{ij},
	\]
	where $a_{ij} \in \K$ and $\omega_{ij} \in \pharm{g}{k}(M; \C)$. Now, since we have already shown $[g]$ to be conformally $\K$-formal in the non-Clifford sense, it follows from Lemma \ref{lem:abs_value_rigidity} that for all indices $i, j \in I$, we have $\abs{\omega_{ij}}_g = C_{ij} \rho^{k}$ for some $C_{ij} \in [0, \infty)$ and a fixed function $\rho \colon M \to [0, \infty)$. Since $\pharm{g}{k}(M; \C)$ is closed under addition, we conclude that
	\[
		\rho^{k-l-m} \dfpart{\alpha \cdot_{g} \beta}_k
		= \sum_{i, j \in I} a_{ij} C_{ij} \omega_{ij}
		\in \pharm{g}{k}(M; \C).
	\]
	
	By again using Lemma \ref{lem:abs_value_rigidity}, we obtain a constant $C \in [0, \infty)$ for which $\smallabs{\rho^{k-l-m} \langle\alpha \cdot_{g} \beta\rangle_k}_g = C \rho^k$. It therefore follows that $\smallabs{\langle\alpha \cdot_{g} \beta\rangle_k}_g = C \rho^{l+m}$. Now, we finally obtain that
	\begin{multline*}
		\dfpart{\alpha \odot_{g} \beta}_k
		= \frac{\dfpart{\alpha \cdot_{g} \beta}_k}
			{\left( C \rho^{l+m} \right)^{\frac{l+m-k}{l+m}}}
		= C^{\frac{k}{l+m} - 1} 
			\left( \rho^{k-l-m} \dfpart{\alpha \cdot_{g} \beta}_k \right)
		\in \pharm{g}{k}(M; \C).
	\end{multline*}
	Hence, we conclude that $\pharm{g}{*}(M; \K)$ is closed under $\odot_{g}$, and therefore $g$ is conformally $\K$-formal in the Clifford sense.
\end{proof}


\section{Hodge theory for $p$-harmonic forms}\label{sect:Hodge_theory}

In this section, we recall several regularity results of $\pharm{g}{k}(M; \K)$ based on non-linear Hodge theory. Our main reference for the results is the paper of Iwaniec, Scott, and Stroffolini \cite{Iwaniec-Scott-Stroffolini}. Note that \cite{Iwaniec-Scott-Stroffolini} is written for $\K = \R$. However, for most of the results we require from \cite{Iwaniec-Scott-Stroffolini}, an exposition of how their proof also works in the case $\K = \C$ can already be found in \cite[Section 4.3]{Kangasniemi_CohomBound}. For the remaining smaller results not covered in \cite{Kangasniemi_CohomBound}, we give an exposition in this chapter on how their proofs also apply when $\K = \C$.

\subsection{Existence and higher integrability results}

Let $M$ be a closed, connected, oriented Riemannian manifold, with Riemannian metric $g_0$. Let $p \in (1, \infty)$, and let $q = p/(p-1)$. A \emph{(nonhomogeneous) Hodge system} of exponent $p$, in the language of \cite{Iwaniec-Scott-Stroffolini}, is a family of differential equations on $M$ of the form
\begin{align}
	d \phi &= 0, \label{eq:ISS_diffeq_d}\\
	d^* \psi &= 0, \label{eq:ISS_diffeq_dstar}\\
	\psi + \psi_0 &= \cG (\phi + \phi_0), \label{eq:ISS_diffeq_conj}
\end{align}
where $\phi_0 \in L^p(\wedge^k M; \K)$, $\psi_0 \in L^q(\wedge^k M; \K)$ are initial data, and $\phi \in L^p(\wedge^k M; \K)$, $\psi \in L^q(\wedge^k M; \K)$ are the solutions. Here, $d^*$ is the usual codifferential $d^* = (-1)^{nk + 1} \hodge_{g_0} d \hodge_{g_0}$, and $\cG \colon L^p(\wedge^k M; \K) \to L^q(\wedge^k M; \K)$ is a point-wise defined measurable operator satisfying three conditions:
\begin{align}
	\abs{\cG (\xi) - \cG (\zeta)}_{g_0} \label{eq:ISS_cond_1}
		&\leq C (\abs{\xi}_{g_0} + \abs{\zeta}_{g_0})^{p-2} 
			\abs{\xi - \zeta}_{g_0},\\
	\Re \ip{\cG (\xi) - \cG (\zeta)}{\xi - \zeta}_{g_0} \label{eq:ISS_cond_2}
		&\geq C^{-1} (\abs{\xi}_{g_0} + \abs{\zeta}_{g_0})^{p-2} 
			\abs{\xi - \zeta}_{g_0}^2,\\
	\cG(t\xi) &= t \abs{t}^{p-2} \cG(\xi) \label{eq:ISS_cond_3},
\end{align}
for $\xi, \zeta \in (\wedge^k T^*_x M)\otimes \K$, $t \in \K$, $x \in M$.

Let $[g]$ be a bounded conformal structure on $M$, and let $k \in \{1, \dots, n-1\}$. The equations of $\pharm{g}{k}(M; \K)$ given in \eqref{eq:diffeq_d}-\eqref{eq:diffeq_dstar} are then an example of a Hodge system for a conformally invariant operator $\cG = \cG_k$, which we define by
\begin{equation}\label{eq:Aharm_ISS_Interpretation}
	\cG_k(\xi) = \hodge_{g_0} \bigl(\abs{\xi}_g^{\frac{n}{k} - 2} \hodge_g \xi\bigr).
\end{equation}

We point out that $\cG_k$ satisfies \eqref{eq:ISS_cond_1}-\eqref{eq:ISS_cond_3} for $p = n/k$, and hence elements of $\pharm{g}{k}(M; \K)$ are solutions $\phi$ of the equations \eqref{eq:ISS_diffeq_d}-\eqref{eq:ISS_diffeq_conj} for data $\phi_0 = \psi_0 = 0$. The fact that that $\cG_k$ satisfies \eqref{eq:ISS_cond_1}-\eqref{eq:ISS_cond_3} follows from the inequalities
\begin{gather}
	\label{eq:ip_p_ineq_1}
	\big\lvert\abs{v}^{p-2}v - \abs{w}^{p-2} w\big\rvert
	\leq C \left( \abs{v} + \abs{w}\right)^{p-2} \abs{v-w} \text{ and}\\
	\label{eq:ip_p_ineq_2}
	\big\langle\abs{v}^{p-2}v - \abs{w}^{p-2} w, v-w \big\rangle
	\geq C^{-1} (\abs{v} + \abs{w})^{p-2} 
	\abs{v-w}^2,
\end{gather}
which hold in every real inner product space $(V, \ip{\cdot}{\cdot})$ for all $v, w \in V$ and $p \in (1, \infty)$, where $C$ depends only on $p$. For complex inner products $\ip{\cdot}{\cdot}$, note that $\Re \ip{\cdot}{\cdot}$ is a real inner product which induces the same norm, and the above estimates may therefore be applied to it. Since $\cG_k$ is conformally invariant, we may assume that \eqref{eq:norm_comparison} holds for $g$, and hence \eqref{eq:ISS_cond_1}-\eqref{eq:ISS_cond_3} follow from \eqref{eq:ip_p_ineq_1}, \eqref{eq:ip_p_ineq_2} applied to $\Re \ip{\cdot}{\cdot}_g$.

We now state the existence and higher integrability result of solutions. For a (long and technical) proof in the case $\K = \R$, see the more general results given in \cite[Theorems 8.4 and 8.8]{Iwaniec-Scott-Stroffolini}. For $\K = \C$, the same proof remains valid with extremely minor changes; see the in-depth discussion in \cite[Section 4.3]{Kangasniemi_CohomBound}. In preparation for the statement, we define the spaces
\begin{align*}
	L^{p, \sharp}(\wedge^k M) 
		&= \bigcup_{q > p} L^{q}(\wedge^k M), \quad \text{and}\\
	L^{p, \flat}(\wedge^k M) 
		&= \bigcap_{q < p} L^{q}(\wedge^k M).
\end{align*}

\begin{thm}\label{thm:ISS_higherint}
	Let $M$ be a closed, connected, oriented Riemannian $n$-mani\-fold with Riemannian metric $g_0$, and let $[g]$ be a bounded measurable conformal structure on $M$. Let $k \in \{1, \dots, n-1\}$, let $\K \in \{\R, \C\}$, and let $\cG_k$ be as in \eqref{eq:Aharm_ISS_Interpretation}. Then for all
	\begin{align*}
		\phi_0 &\in L^{\frac{n}{k}}(\wedge^k M; \K) &\text{and}&&
		\psi_0 &\in L^{\frac{n}{n-k}}(\wedge^k M; \K),
	\end{align*}
	there exist unique
	\begin{align*}
		\phi &\in L^{\frac{n}{k}}(\wedge^k M; \K) &\text{and}&&
		\psi &\in L^{\frac{n}{n-k}}(\wedge^k M; \K),
	\end{align*}
	for which $\phi = d\alpha$ for some $\alpha \in L^{n/k}(\wedge^{k-1} M; \K)$ and $(\phi, \psi)$ satisfy \eqref{eq:ISS_diffeq_d}-\eqref{eq:ISS_diffeq_conj} for data $\phi_0, \psi_0, \cG_k$.
	
	If, moreover, $\phi_0 \in L^{n/k, \sharp}(\wedge^k M; \K)$ and $\psi_0 \in L^{n/(n-k), \sharp}(\wedge^k M; \K)$, then also $\phi \in L^{n/k, \sharp}(\wedge^k M; \K)$ and $\psi \in L^{n/(n-k), \sharp}(\wedge^k M; \K)$.
\end{thm}

We now apply Theorem \ref{thm:ISS_higherint} to obtain several properties of $\pharm{g}{k}(M; \K)$. The first consequence is an existence property for elements of $\pharm{g}{k}(M; \K)$.

\begin{lemma}\label{lem:harmonic_existence}
	Let $M$ be a closed, connected, oriented Riemannian $n$-mani\-fold with Riemannian metric $g_0$, and let $[g]$ be a bounded measurable conformal structure on $M$. Let $k \in \{1, \dots, n-1\}$, and let $\K \in \{\R, \C\}$. Then for every $k$-form $\eta \in L^{n/k}(\wedge^k M; \K)$ with $d\eta = 0$ weakly, there exists a unique $\omega \in \pharm{g}{k}(M; \K)$ for which $\eta - \omega = d\tau$ for some $\tau \in L^{n/k}(\wedge^{k-1} M; \K)$.
\end{lemma}
\begin{proof}
	We apply Theorem \ref{thm:ISS_higherint} with $\phi_0 = \eta$ and $\psi_0 = 0$. Now, $(\phi, \psi)$ is a solution if and only if $\phi + \phi_0 \in \pharm{g}{k}(M; \K)$. We therefore have $\omega = \eta + \phi$ for the unique solution $(\phi, \psi)$.
\end{proof}

For the second property, we recall the \emph{$L^p$ Hodge decomposition}. Namely, suppose that $\omega \in L^p(\wedge^k M; \K)$ for some $k \in \N$ and $1 < p < \infty$. Then $\omega$ can be uniquely written as the sum of three forms $\alpha, \beta, \gamma \in L^p(\wedge^k M; \K)$ with specific properties. First, $\alpha$ is the weak differential $\alpha = d\xi$ of some form $\xi \in L^p(\wedge^{k-1} M; \K)$. Second, $\beta$ is the weak codifferential $\beta = d^* \zeta$ of some form $\zeta \in L^p(\wedge^{k+1} M; \K)$. Third, we have $d\gamma = d^*\gamma = 0$ in the weak sense. In particular, $\gamma$ is harmonic, and therefore smooth.

For a proof of the $L^p$ Hodge decomposition, see \cite[Proposition 6.5]{Scott_LpHodge} or \cite[Theorem 5.7]{Iwaniec-Scott-Stroffolini}. While the proofs are for $\K = \R$, the case $\K = \C$ can in fact be reduced to the real case, since $L^p(\wedge^k M; \C) = L^p(\wedge^k M; \R) \otimes \C$ and the operators $d$ and $d^*$ are linear. Note also that if $d\omega = 0$, then in the decomposition for $\omega$ we have $\beta = 0$. Indeed, otherwise we could also decompose $\omega = \alpha + 0 + (\beta+\gamma)$, which would violate the uniqueness of the decomposition.

\begin{lemma}\label{lem:p_harm_higher_int}
	Let $M$ be a closed, connected, oriented Riemannian $n$-mani\-fold with Riemannian metric $g_0$, and let $[g]$ be a bounded measurable conformal structure on $M$. Let $k \in \{1, \dots, n-1\}$, and let $\K \in \{\R, \C\}$. Then $\pharm{g}{k}(M; \K) \subset L^{n/k, \sharp}(\wedge^k M; \K)$.
\end{lemma}
\begin{proof}
	Let $\omega \in \pharm{g}{k}(M; \K)$. We take the $L^{n/k}$ Hodge decomposition $\omega = \alpha + \beta + \gamma$. Since $d\omega = 0$, we have $\beta = 0$ as discussed before.
	
	We now apply Theorem \ref{thm:ISS_higherint} for $\phi_0 = \gamma$ and $\psi_0 = 0$. Now, $(\phi, \psi) = (\alpha, \cG_k(\omega))$ is a valid solution, so by uniqueness it is the only one. Since $\gamma$ is harmonic and therefore smooth, it is in $L^{n/k, \sharp}(\wedge^k M; \K)$. Hence, we obtain that $\alpha = \phi \in L^{n/k, \sharp}(\wedge^k M; \K)$, and therefore $\omega = \alpha + \gamma \in L^{n/k, \sharp}(\wedge^k M; \K)$.
\end{proof}

Finally, using the properties shown so far, we give a version of Theorem \ref{thm:ISS_higherint} which in fact more closely resembles the $L^p$ Hodge decomposition.

\begin{prop}\label{prop:nonlinear_Hodge_decomposition}
	Let $M$ be a closed, connected, oriented Riemannian $n$-mani\-fold with Riemannian metric $g_0$, let $[g]$ be a bounded measurable conformal structure on $M$, and let $g \in [g]$. Let $k \in \{1, \dots, n-1\}$ and $\K \in \{\R, \C\}$. Then we may express every $k$-form $\omega \in L^{n/k}(\wedge^k M; \K)$ in the form
	\[
		\omega = d\alpha + \abs{d\beta}_g^{\frac{n}{n-k} - 2} \hodge_g d\beta + \gamma,
	\]
	where $d\alpha \in L^{n/k}(\wedge^k M; \K)$, $d\beta \in L^{n/(n-k)}(\wedge^{n-k} M; \K)$, and $\gamma \in \pharm{g}{k}(M; \K)$ are all unique. Moreover, if $\omega \in L^{n/k, \sharp}(\wedge^k M; \K)$, then we also have $d\alpha \in L^{n/k, \sharp}(\wedge^k M; \K)$ and $d\beta \in L^{n/(n-k), \sharp}(\wedge^{n-k} M; \K)$.
\end{prop}
\begin{proof}
	We apply Theorem \ref{thm:ISS_higherint} with $\psi_0 = \hodge_{g_0} \omega \in L^{n/k}(\wedge^{n-k}M; \K)$ and $\phi_0 = 0$. We obtain
	\[
		\omega = \abs{\phi}_g^{\frac{n}{n-k} - 2} \hodge_g \phi - (-1)^{k(n-k)} \hodge_{g_0} \psi.
	\]
	Here $\phi = d\beta \in L^{n/(n-k)}(\wedge^{n-k} M; \K)$ and $\psi \in L^{n/k}(\wedge^{n-k} M; \K) \cap \ker d^*$ are unique. Moreover, if $\omega \in L^{n/k, \sharp}(\wedge^k M; \K)$, then $d\beta \in L^{n/(n-k), \sharp}(\wedge^{n-k} M; \K)$ and $\psi \in L^{n/k, \sharp}(\wedge^{n-k} M; \K)$.
	
	Finally, since $\psi \in \ker d^*$, we have $\hodge_{g_0} \psi \in \ker d$. Hence, we may use Lemma \ref{lem:harmonic_existence} on $\hodge_{g_0} \psi$, and obtain that $\hodge_{g_0} \psi = d\alpha + \gamma$ for some unique $\gamma \in \pharm{g}{k}(M; \K)$. Moreover, if $\omega \in L^{n/k, \sharp}(\wedge^k M; \K)$, then by Lemma \ref{lem:p_harm_higher_int} we have $d\alpha = \hodge_{g_0} \psi - \gamma \in L^{n/k, \sharp}(\wedge^k M; \K)$. The claim follows.
\end{proof}

For a $\omega \in L^{n/k}(\wedge^k M)$ with $k \in \{1, \dots, n-1\}$, and for a bounded measurable conformal structure $[g]$ on $M$, we call the unique triple $(d\alpha, d\beta, \gamma)$ provided by Proposition \ref{prop:nonlinear_Hodge_decomposition} the \emph{conformal Hodge decomposition of $\omega$ with respect to $[g]$}. We point out that the decomposition indeed depends only on the conformal structure.

\begin{lemma}\label{lem:conformal_decomposition_indep}
	Let $M$ be a closed, connected, oriented Riemannian $n$-mani\-fold, and let $[g]$ be a bounded measurable conformal structure on $M$. Let $\omega \in L^{n/k}(\wedge^k M; \K)$, where $k \in \{1, \dots, n-1\}$ and $\K \in \{\R, \C\}$. Then the conformal Hodge decomposition $(d\alpha, d\beta, \gamma)$ of $\omega$ provided by Proposition \ref{prop:nonlinear_Hodge_decomposition} does not depend on the choice of $g \in [g]$.
\end{lemma}
\begin{proof}
	Let $\rho \colon M \to (0, \infty)$ be measurable, and let $(d\alpha, d\beta, \gamma)$ be the conformal Hodge decomposition of $\omega$ obtained using the metric $g \in [g]$. By uniqueness, it suffices to show that $(d\alpha, d\beta, \gamma)$ is also a conformal Hodge decomposition of $\omega$ for $\rho^2 g \in [g]$. However, this is true, since $\gamma \in \pharm{g}{k}(M; \K) = \pharm{\rho^2 g}{k}(M; \K)$ and since
	\begin{multline*}
		\abs{d\beta}_{\rho^2 g}^{\frac{n}{n-k} - 2} \hodge_{\rho^2 g} d\beta
		= \left(\rho^{-(n-k)} \abs{d\beta}_{g} \right)^{\frac{n}{n-k} - 2}
			\rho^{n - 2(n-k)} \hodge_{g} d\beta\\
		= \rho^{(2(n-k) - n) + (n - 2(n-k))} 
			\abs{d\beta}_{g}^{\frac{n}{n-k} - 2} \hodge_{g} d\beta
		= \abs{d\beta}_{g}^{\frac{n}{n-k} - 2} \hodge_{g} d\beta.
	\end{multline*}
\end{proof}

\subsection{Continuity}

The final regularity result we require is a continuity property associated with Lemma \ref{lem:harmonic_existence}. For $\K = \R$, it is a special case of \cite[Theorem 8.5]{Iwaniec-Scott-Stroffolini}, with the essential ideas of its proof explained in \cite[Proposition 7.1]{Iwaniec-Scott-Stroffolini}. However, for the convenience of the reader, we present a proof which also takes into account the case $\K = \C$.

\begin{thm}\label{thm:ISS_continuity_complex}
	Let $M$ be a closed, connected, oriented Riemannian $n$-mani\-fold with Riemannian metric $g_0$, and let $[g]$ be a bounded measurable conformal structure on $M$. Let $k \in \{1, \dots, n-1\}$, and let $\K \in \{\R, \C\}$. Then there exist constants $C \in (0, \infty)$ and $t \in (0, 1)$ with the following property: for all $k$-forms $\eta, \eta' \in L^{n/k}(\wedge^k M; \K)$ with $d\eta = d \eta' = 0$ weakly, and for the corresponding unique $\omega, \omega' \in \pharm{g}{k}(M; \K)$ provided by Lemma \ref{lem:harmonic_existence}, we have
	\[
		\norm{\omega - \omega'}_{n/k, g} \leq 
		C \left(\norm{\omega}_{n/k, g} + \norm{\omega'}_{n/k, g}\right)^{1-t}
		\norm{\eta - \eta'}_{n/k, g}^t.
	\]
\end{thm}
\begin{proof}
	We may assume that $g \in [g]$ satisfies $\vol_g = \vol_{g_0}$, and therefore also that \eqref{eq:norm_comparison} holds. We have $\eta - \omega = d\tau$ and $\eta' - \omega' = d\tau'$ for some $\tau, \tau' \in L^{n/k}(\wedge^k M; \K)$. We claim that
	\begin{multline}\label{eq:ISS_monstrosity_sameintegrals}
		\int_M \ip{\omega - \omega'}
		{\abs{\omega}_g^{\frac{n}{k} - 2} \omega -
			\abs{\omega'}_g^{\frac{n}{k} - 2} \omega'}_g\vol_g\\
		= \int_M \ip{\eta - \eta'}
		{\abs{\omega}_g^{\frac{n}{k} - 2} \omega -
			\abs{\omega'}_g^{\frac{n}{k} - 2} \omega'}_g\vol_g.
	\end{multline} 
	Indeed, this follows from the computation
	\begin{align*}
		&\int_M \ip{(\eta - \omega) - (\eta' - \omega')}
			{\abs{\omega}_g^{\frac{n}{k} - 2} \omega -
				\abs{\omega'}_g^{\frac{n}{k} - 2} \omega'}_g
			\vol_g\\
		&\quad= \int_M (d\tau - d\tau')\wedge 
			\bigl(\abs{\omega}_g^{\frac{n}{k} - 2} \hodge_g\omega -
				\abs{\omega'}_g^{\frac{n}{k} - 2} \hodge_g\omega'\bigr)\\
		&\quad= \int_M (\tau - \tau')\wedge 
			d\bigl(\abs{\omega}_g^{\frac{n}{k} - 2} \hodge_g\omega -
			\abs{\omega'}_g^{\frac{n}{k} - 2} \hodge_g\omega'\bigr)\\
		&\quad= \int_M (\tau - \tau') \wedge 0 = 0.
	\end{align*}

	The first part of the proof is to obtain an estimate for the integral of $(\abs{\omega}_{g} + \abs{\omega'}_{g})^{\frac{n}{k}-2} 
	\abs{\omega - \omega'}_{g}^2$. By using \eqref{eq:ip_p_ineq_2} and \eqref{eq:ISS_monstrosity_sameintegrals} we obtain
	\begin{multline*}
		\int_M (\abs{\omega}_{g} + \abs{\omega'}_{g})^{\frac{n}{k}-2} 
			\abs{\omega - \omega'}_{g}^2 \vol_g\\
		\leq C \int_M \Re \ip{\abs{\omega}_g^{\frac{n}{k} - 2} \omega -
			\abs{\omega'}_g^{\frac{n}{k} - 2}\omega'}{\omega - \omega'}_{g}
			\vol_g\\
		\leq C \abs{\int_M \ip{\abs{\omega}_g^{\frac{n}{k} - 2} \omega -
			\abs{\omega'}_g^{\frac{n}{k} - 2}\omega'}{\omega - \omega'}_{g}
			\vol_g}\\
		= C \abs{\int_M \ip{\abs{\omega}_g^{\frac{n}{k} - 2} \omega -
			\abs{\omega'}_g^{\frac{n}{k} - 2}\omega'}{\eta - \eta'}_{g}
			\vol_g}.
	\end{multline*}
	Hence, a use of H\"older's inequality yields
	\begin{align*}
		&\int_M (\abs{\omega}_{g} + \abs{\omega'}_{g})^{\frac{n}{k}-2} 
			\abs{\omega - \omega'}_{g}^2 \vol_g\\
		&\hspace{1.5cm}\leq C \norm{\eta - \eta'}_{\frac{n}{k}, g}
			\norm{\abs{\omega}_g^{\frac{n}{k} - 2} \omega -
			\abs{\omega'}_g^{\frac{n}{k} - 2}\omega'}_{\frac{n}{n-k}, g}\\
		&\hspace{1.5cm}\leq C \norm{\eta - \eta'}_{\frac{n}{k}, g}
			\left(
			\norm{\abs{\omega}_g^{\frac{n}{k} - 2} \omega}_{\frac{n}{n-k}, g}
			+ \norm{\abs{\omega'}_g^{\frac{n}{k} - 2}\omega'}_{\frac{n}{n-k}, g}
			\right)\\
		&\hspace{1.5cm}= C \norm{\eta - \eta'}_{\frac{n}{k}, g}
			\left(
			\norm{\omega}_{\frac{n}{k}, g}^\frac{n-k}{k}
			+ \norm{\omega'}_{\frac{n}{k}, g}^\frac{n-k}{k}
			\right)\\
		&\hspace{3cm}\leq 2C \norm{\eta - \eta'}_{\frac{n}{k}, g}
			\left(
			\norm{\omega}_{\frac{n}{k}, g}
			+ \norm{\omega'}_{\frac{n}{k}, g}
			\right)^\frac{n-k}{k}		
	\end{align*}
	
	We then let $s = \max(k, n-k)$, and $u = 2s/n$. It follows that $u > 1$, and 
	\[
		\frac{n}{k} - \frac{2}{u} = \frac{n}{k} - \frac{n}{s} \geq 0. 
	\]
	Hence, we may now estimate by the triangle inequality that
	\begin{multline*}
		\abs{\omega - \omega'}_g^\frac{n}{k}
		\leq \left(\abs{\omega}_g + \abs{\omega'}_g
			\right)^{\frac{n}{k} - \frac{2}{u}}
			\abs{\omega - \omega'}_g^\frac{2}{u}\\
		= \left(\abs{\omega}_g + \abs{\omega'}_g
				\right)^{\frac{1}{u}\left(\frac{n}{k} - 2\right)}
			\abs{\omega-\omega'}^\frac{2}{u} \cdot
			\left(\abs{\omega}_g + \abs{\omega'}_g
				\right)^{\frac{n}{k}\cdot \frac{u-1}{u}}.		
	\end{multline*}
	Another use of H\"older's inequality yields
	\begin{multline*}
		\norm{\omega-\omega'}_{\frac{n}{k}, g}
		\leq \left( \int_M (\abs{\omega}_{g} + \abs{\omega'}_{g})^{\frac{n}{k}-2} 
			\abs{\omega - \omega'}_{g}^2 \vol_g \right)^\frac{k}{nu}\\
		\left( \int_M \left(\abs{\omega}_g + \abs{\omega'}_g
			\right)^{\frac{n}{k}} \vol_g \right)^\frac{k(u-1)}{nu}
	\end{multline*}
	By chaining our estimates, and using the triangle inequality of the $L^p$ norm, we obtain
	\[
		\norm{\omega-\omega'}_{\frac{n}{k}, g}\\
		\leq 2C \norm{\eta - \eta'}_{\frac{n}{k}, g}^\frac{k}{nu}
		\left(\norm{\omega}_{\frac{n}{k}, g}
			+ \norm{\omega'}_{\frac{n}{k}, g}\right)^{\frac{n-k}{nu} + \frac{u-1}{u}}.
	\]
	The claim hence follows.
\end{proof}

\section{Conformal cohomology and the embedding theorem}\label{sect:conf_cohomology}

Our goal in this section is to prove Theorems \ref{thm:conf_formal_algebra_embedding} and \ref{thm:conf_formal_clifford_embedding}. For this, the key tool we use is conformal cohomology. We first discuss our choice of conformal cohomology theory, and recall the two key properties we require: isomorphism with singular cohomology, and unique representation by elements of $\pharm{g}{*}(M; \K)$. With these tools, we then complete the proof of the cohomological embedding theorem for conformally formal manifolds.

\subsection{Cohomology in the conformal exponent}

Let $M$ be a closed, connected, oriented Riemannian $n$-manifold, with Riemannian metric $g_0$. The idea of conformal cohomology is to consider a chain complex consising of measurable $k$-forms $\omega \in L^{n/k}_\loc(\wedge^k M; \K)$ which have a weak differential $d\omega \in L^{n/(k+1)}_\loc(\wedge^{k+1} M; \K)$. The cohomology of this complex is very close to the usual singular and de Rham cohomologies. However, a difference occurs at the endpoints in the complex, due to the failure of Poincar\'e inequalities for the extremal exponents 1 and $\infty$. 

Currently, several approaches exist in the literature to modifying the aforementioned chain complex so that this difference with usual de Rham cohomology is eliminated. We use the beginning of this section to briefly survey some such approaches. Note that we present the global versions of the chain complexes, since we've restricted our attention to closed manifolds $M$; for non-compact $M$, all integrability conditions would have to be replaced with local versions.

In \cite{Kangasniemi-Pankka_PLMS}, the author and Pankka present an approach which modifies only the ends of the complex. The corresponding chain complex, $\cesob(\wedge^* M; \K)$, is defined by
\begin{align*}
	\cesob(\wedge^0 M; \K) 
		&= \left\{ \omega \in L^{\infty, \flat}(\wedge^0 M; \K) :
			d\omega \in L^n(\wedge^1 M; \K) \right\},\\
	\cesob(\wedge^k M; \K) 
		&= \left\{ \omega \in L^{\frac{n}{k}}(\wedge^k M; \K) :
			d\omega \in L^{\frac{n}{k+1}}(\wedge^{k+1} M; \K) \right\} \\
		&\hspace{5.5cm}\text{ for } 1 \leq k \leq n-2,\\
	\cesob(\wedge^{n-1} M; \K) 
		&= \left\{ \omega \in L^{\frac{n}{n-1}}(\wedge^{n-1} M; \K) :
			d\omega \in L^{1, \sharp}(\wedge^n M; \K) \right\},\\
	\cesob(\wedge^{n} M; \K) &= L^{1, \sharp}(\wedge^n M; \K),
\end{align*}
where the differentials are defined weakly. The weak differential $d$ is a chain map for the complex $\cesob(\wedge^* M)$, and the $k$:th cohomology of this complex is denoted $\cehom{k}(M)$.

The cohomology spaces $\cehom{k}(M; \K)$ are naturally isomorphic with the singular cohomology spaces $H^k(M; \K)$; see \cite[Section 4]{Kangasniemi-Pankka_PLMS} and \cite[(3.1)]{Kangasniemi_CohomBound}. The main advantage of this complex $\cesob(\wedge^* M; \K)$ is that every cohomology class $[\omega] \in \cehom{k}(M; \K)$ for $0 < k < n$ is complete under the conformally invariant norm $\norm{\cdot}_{n/k, g_0}$; see \cite[Lemma 3.3]{Kangasniemi-Pankka_PLMS} and \cite[Lemma 3.2]{Kangasniemi_CohomBound}. A major disadvantage, however, is that $\cesob(\wedge^* M; \K)$ is not closed under the wedge product. Hence, we do not inherit a natural ring structure in $\cehom{*}(M)$ from the complex $\cesob(\wedge^* M; \K)$.

Two other approaches are given in \cite{Donaldson-Sullivan_Acta} by Donaldson and Sullivan. For the first one, the complex is given by
\begin{align*}
	\cesobs(\wedge^0 M; \K) 
		&= \left\{ \omega \in C(\wedge^0 M; \K) :
		d\omega \in L^{n, \sharp}(\wedge^1 M; \K) \right\},\\
	\cesobs(\wedge^k M; \K) 
		&= \left\{ \omega \in L^{\frac{n}{k}, \sharp}(\wedge^k M; \K) :
		d\omega \in L^{\frac{n}{k+1}, \sharp}(\wedge^{k+1} M; \K) \right\} \\
		&\hspace{5.5cm}\text{ for } 1 \leq k \leq n-1,\\
	\cesobs(\wedge^{n} M; \K) &= L^{1, \sharp}(\wedge^n M; \K),
\end{align*}
where $C(\wedge^0 M; \K)$ denotes the space of continuous 0-forms on $M$, and the differentials are again interpreted weakly. We denote the resulting cohomology by $\cehoms{*}(M; \K)$. The cohomology $\cehoms{*}(M; \K)$ is naturally isomorphic to $H^*(M; \K)$, and now a natural wedge product structure exists by $[\omega] \wedge [\tau] = [\omega \wedge \tau]$ for $\omega, \tau \in \cesobs(\wedge^* M; \K) \cap \ker d$. The main disadvantage of the complex is the lack of completeness properties with respect to any norm.

The other approach of Donaldson and Sullivan is by a more complicated norm on $k$-forms, where one fixes parameters $\rho, \eps \in (0,1)$ and defines a norm
\[
	\!\!\norm{\omega}_{p, \rho, \eps} 
	= \inf \left\{
		\sqrt{\sum_{i=1}^\infty \rho^{-i} \norm{f_i}_{p+\eps^i}^2} :
		f_i \in L^{p+\eps^i}(M), 
		\norm{ \abs{\omega}_{g_0} - \sum_{i=1}^j f_i}_p \xrightarrow[j \to \infty]{} 0
	\right\}
\]
for $p \in [1, \infty)$ and $\omega \in L^p(\wedge^k M; \K)$. This defines a subspace $L^p_{\rho, \eps}(\wedge^k M; \K) \subset L^p(\wedge^k M; \K)$ which is complete under $\norm{\omega}_{p, \rho, \eps}$. Then the complex defined by
\begin{align*}
	\cesobds(\wedge^0 M; \K) 
		&= \left\{ \omega \in C(\wedge^0 M; \K) :
		d\omega \in L^{n}_{\rho, \eps}(\wedge^1 M; \K) \right\},\\
	\cesobds(\wedge^k M; \K) 
		&= \left\{ \omega \in L^{\frac{n}{k}}_{\rho, \eps}(\wedge^k M; \K) :
		d\omega \in L^{\frac{n}{k+1}}_{\rho, \eps}(\wedge^{k+1} M; \K) \right\} \\
		&\hspace{5cm}\text{ for } 1 \leq k \leq n-1,\\
	\cesobds(\wedge^{n} M; \K) &= L^{1}_{\rho, \eps}(\wedge^n M; \K)
\end{align*}
yields a cohomology space naturally isomorphic to $H^*(M; \K)$. We denote this cohomology space by $\cehomds{*}(M; \K)$. The space $\cehomds{*}(M; \K)$ provides both a natural wedge product in cohomology and a Banach space structure, at the cost of a parameter-dependent norm which is only quasi-preserved under conformal maps.

\medskip

In this paper, we mainly use the space $\cehoms{*}(M; \K)$. It has two important properties for us: elements of $\cesobs(\wedge^0 M; \K)$ are continuous, and $\cehoms{*}(M; \K)$ inherits a natural wedge product structure from the chain complex. The lack of these properties for $\cehom{*}(M; \K)$ makes it worse suited for the results of this paper. Moreover, the extra Banach space properties of $\cehomds{*}(M; \K)$ are unnecessary for our results, so we choose $\cehoms{*}(M; \K)$ over it to avoid the  complications involved with using $L^p_{\rho, \eps}$-norms.

As discussed before, the cohomology space $\cehoms{*}(M; \K)$ is isomorphic to $H^*(M; \K)$. Indeed, this is stated in \cite[Proposition 4.2]{Donaldson-Sullivan_Acta} in the case $n = 4$ with real coefficients, referring to sheaf-theoretic arguments for the proof. A more detailed explanation of such sheaf-theoretic arguments can be found in \cite[Section 4]{Kangasniemi-Pankka_PLMS}, where they are used to show that $\cehom{*}(M; \R)$ is isomorphic to $H^*(M; \R)$ as a graded vector space.

However, for the convenience of the reader, we provide an alternative proof using the $L^p$ Hodge decomposition, which reduces the question to the similar isomorphism result for smooth de Rham cohomology $\derham^{*}(M; \K)$. Note that the proof can also be used for $\cehom{*}(M; \K)$ with minor changes, providing a possible alternative for the discussion in \cite[Section 4]{Kangasniemi-Pankka_PLMS} in the case of closed manifolds.

\begin{lemma}\label{lem:DS_algebra_isom}
	Let $M$ be a closed, connected, oriented Riemannian $n$-manifold, and let $\K \in \{\R, \C\}$. Then the map $\derham^{*}(M; \K) \to \cehoms{*}(M; \K)$ induced by the inclusion of chain complexes $C^\infty(\wedge^* M; \K) \hookrightarrow \cesobs(\wedge^* M; \K)$ is an isomorphism of algebras.
\end{lemma}
\begin{proof}
	Let $\iota$ denote the map $\derham^*(M; \K) \to \cehoms{*}(M; \K)$ induced by the inclusion maps $C^\infty(\wedge^* M; \K) \hookrightarrow \cesobs(\wedge^* M; \K)$. Since the inclusion maps are linear and commute with the boundary maps $d$, the map $\iota$ is well defined and linear. Moreover, since the inclusions respect wedge products, the map $\iota$ is a homomorphism of algebras. It remains to show that $\iota$ is bijective.
		
	For surjectivity, the argument is essentially as in e.g.\ \cite[Lemma 3.3]{Kangasniemi_CohomBound}. Suppose $[\omega] \in \cehoms{k}(M; \K)$, with the goal of finding a $\omega' \in C^\infty(\wedge^k M; \K)$ with $[\omega'] = [\omega]$. Consider first the case $k > 1$. Then $\omega \in L^p(\wedge^k M; \K)$ for some $p > n/k \geq 1$. Since $d\omega = 0$, we obtain by the $L^p$ Hodge decomposition that $\omega = d\tau + \gamma$, where $\tau \in L^p(\wedge^{k-1} M; \K)$ and $\gamma$ is harmonic. By the Sobolev--Poincar\'e inequality, we may assume $\tau \in \cesobs(\wedge^{k-1} M; \K)$; see e.g. \cite{Goldshtein-Troyanov_Sobolev} for the case $\K = \R$, which also easily implies the case $\K = \C$. Hence, $\gamma$ is a smooth form contained in $[\omega]$, concluding the case. 
	
	The case $k = 1$ is done as above, by decomposing $\omega = d\tau + \gamma$ and showing that $\tau \in \cesobs(\wedge^{0} M; \K)$. However, instead of using the aforementioned Sobolev--Poincar\'e inequality to show this, we note that the 0-form $\tau$ is a Sobolev function in a space $W^{1,p}(M; \K)$ with $p > n$. Therefore $\tau$ is continuous by the Sobolev embedding theorem, and hence $\tau \in \cesobs(\wedge^0 M; \K)$. Finally, for the remaining case $k = 0$, we note that Sobolev functions with zero weak gradient are locally a.e.\ constant, see e.g. \cite[Lemma 1.16]{Heinonen-Kilpelainen-Martio_book}. Therefore, $\cesobs(\wedge^0 M; \K) \cap \ker d = C^\infty(\wedge^0 M; \K) \cap \ker d$ as they both consist of constant functions, and the remaining case $k = 0$ of surjectivity is thus complete.
		
	It remains to show injectivity. Let $\omega \in C^\infty(\wedge^k M; \K) \cap \ker d$ be such that $\omega = d\tau$ for some $\tau \in \cesobs(\wedge^{k-1} M; \K)$. Then we may take the classical (smooth) Hodge decomposition of $\omega$: $\omega = \alpha + \beta + \gamma$ where $\alpha \in dC^\infty(\wedge^{k-1} M; \K)$, $\beta \in d^*C^\infty(\wedge^{k-1} M; \K)$, and $\gamma$ is harmonic. However, now $\omega = d\tau$ and $\omega = \alpha + \beta + \gamma$ are two $L^p$ Hodge decompositions of $\omega$. By uniqueness, we conclude that $\beta = \gamma = 0$, and therefore $\omega \in dC^\infty(\wedge^{k-1} M; \K)$. Injectivity of $\iota$ follows, concluding the proof.
\end{proof}

\subsection{The $p$-harmonic representation}
Suppose then that $[g]$ is a bounded conformal structure on $M$. We now recall how, for $k \in \{0, \dots, n-1\}$, $\cehoms{k}(M; \K)$ is uniquely represented by $\pharm{g}{k}(M; \K)$. The result follows rather quickly from the facts presented in Section \ref{sect:Hodge_theory}. 

\begin{lemma}\label{lem:cohom_repr}
	Let $M$ be a closed, connected, oriented, Riemannian $n$-mani\-fold, let $[g]$ be a bounded conformal structure on $M$, and let $\K \in \{\R, \C\}$. Then for every $k \in \{0, \dots, n-1\}$ we have $\pharm{g}{k}(M; \K) \subset \cesobs(\wedge^k M; \K)$. Moreover, the map $\pharm{g}{k}(M; \K) \to \cehoms{k}(M; \K)$ which maps $\omega \in \pharm{g}{k}(M; \K)$ to its cohomology class $[\omega]$ is bijective.
\end{lemma}
\begin{proof}
	The case $k = 0$ is clear since $\pharm{g}{0}(M; \K)$ consists of only the constant functions, and every cohomology class of $\cehoms{0}(M; \K)$ is a singleton of a constant function.
	
	Let now $k \in \{1, \dots, n-1\}$, and let $\omega \in \pharm{g}{k}(M; \K)$. By Lemma \ref{lem:p_harm_higher_int}, we have $\omega \in L^{n/k, \sharp}(\wedge^k M; \K)$. Moreover, we have $d\omega = 0 \in L^{n/(k+1), \sharp}(\wedge^k M; \K)$. Hence, $\omega \in \cesobs(\wedge^k M; \K)$, and it belongs to a cohomology class $[\omega] \in \cehoms{k}(M; \K)$.
	
	Suppose then that $c$ is a cohomology class in $\cehoms{k}(M; \K)$. By Lemma \ref{lem:DS_algebra_isom}, $c = [\omega_0]$ for some smooth element $\omega_0 \in C^\infty(\wedge^k M; \K)$. By Lemma \ref{lem:harmonic_existence}, there exists a unique $\omega \in \pharm{g}{k}(M; \K)$ such that $\omega - \omega_0 = d\tau$. Since $\omega, \omega_0 \in L^{n/k, \sharp}(\wedge^k M; \K)$, we obtain similarly as in the proof of Lemma \ref{lem:DS_algebra_isom} that $d\tau \in d\cesobs(\wedge^{k-1} M; \K)$, by using either the Sobolev-Poincar\'e inequality if $k > 1$ or the Sobolev embedding theorem is $k = 1$. Hence, there exists a unique $\omega \in c \cap \pharm{g}{k}(M; \K)$, which shows the bijectivity part of the claim.
\end{proof}

We then separately consider the special case $\pharm{g}{n}(M; \K)$, which we have only defined for conformally formal $[g]$.

\begin{lemma}\label{lem:cohom_repr_n}
	Let $M$ be a closed, connected, oriented, Riemannian $n$-mani\-fold, and let $[g]$ be a conformally $\K$-formal bounded conformal structure on $M$, where $\K \in \{\R, \C\}$. Suppose that $\pharm{g}{k}(M; \K) \neq \{0\}$ for some $k \in \{1, \dots, n-1\}$. Then $\pharm{g}{n}(M; \K) \subset \cesobs(\wedge^k M)$, and the map $\pharm{g}{n}(M; \K) \to \cehoms{n}(M; \K)$ which maps $\omega \in \pharm{g}{n}(M; \K)$ to its cohomology class $[\omega]$ is bijective.
\end{lemma}
\begin{proof}
	Let $g_0$ be a smooth Riemannian metric on $M$, in which case we have $d([g], [g_0]) < \infty$. Let $g \in [g]$ be such that $\vol_g = \vol_{g_0}$. By our assumptions, there exists an $\omega' \in \pharm{g}{k}(M; \K) \setminus \{0\}$ for some $k \in \{1, \dots, n-1\}$. By Lemma \ref{lem:Hn_representation}, elements of $\pharm{g}{n}(M; \K)$ are of the form $\omega = C \abs{\omega'}_g^{n/k} \vol_g$ for $C \in \R$. Since $\omega' \in L^{n/k, \sharp}(\wedge^k M; \K)$ by Lemma \ref{lem:p_harm_higher_int}, we obtain by \eqref{eq:norm_comparison} that $\omega \in L^{1, \sharp}(\wedge^n M; \K) = \cesobs(\wedge^n M; \K)$.
	
	Hence, we have a well-defined map $\pharm{g}{n}(M; \K) \to \cehoms{n}(M; \K)$ by $\omega \mapsto [\omega]$. Since both $\pharm{g}{n}(M; \K)$ and $\cehoms{n}(M; \K) \cong H^n(M; \K)$ are 1-dimensional $\K$-vector spaces, it suffices to show that this map is injective. Suppose then that $\omega \in \pharm{g}{n}(M; \K)$ and $[\omega] = [0]$. Then $\omega = d\tau$ for some $\tau \in \cesobs(\wedge^{n-1} M; \K)$. By a measurable Stokes' theorem, it follows that $\int_M \omega = 0$. However, since $\omega \in \pharmgeq{g}{n}(M; \K)$, we have that $\hodge_g \omega$ remains on a fixed half-line starting from 0. Hence, $\int_M \omega = 0$ is only possible if $\omega = 0$, and the proof is concluded.
\end{proof}

\subsection{The embedding theorem}

We are now ready to prove Theorems \ref{thm:conf_formal_algebra_embedding} and \ref{thm:conf_formal_clifford_embedding}, providing our main topological obstruction to the conformal formality of a manifold. We first recall the statements of the theorems by giving a version which includes the case $\K = \C$, as well as both the wedge product and Clifford product cases.

\begin{thm}\label{thm:conf_formal_algebra_embedding_C}
	Let $M$ be a closed, connected, oriented, smooth $n$-manifold. Suppose that $M$ is conformally $\K$-formal, where $\K \in \{\R, \C\}$. Then there exists an embedding of graded algebras $\Phi \colon H^*(M; \K) \to \wedge^* \K^n$ which maps the cup product to the exterior product. The map $\Phi$ can be selected as such that its image is closed under the Hodge star. Moreover, if $M$ is conformally $\K$-formal in the Clifford sense, then $\Phi$ can also be selected as such that its image is closed under the Euclidean Clifford product of $\wedge^* \K^n$.
\end{thm} 

The argument is essentially the same as in \cite{Kangasniemi_CohomBound}, but additional attention has been given to the extra product structure provided by our assumption. Lemma \ref{lem:abs_value_rigidity} gives us the necessary information to perform the argument.

\begin{proof}
	We may assume $H^k(M; \K) \neq \{0\}$ for some $k \in \{1, \dots, n-1\}$. Indeed, otherwise $H^*(M; \K) \cong H^*(\S^n; \K)$ which embeds into $\wedge^* \K^n$ in the desired way. Let $[g]$ be a conformally $\K$-formal bounded conformal structure on $M$.
	
	By classical de Rham theory, we have $\derham^*(M; \K) \cong H^*(M; \K)$ as algebras, where in this isomorphism, a smooth de Rham cohomology class $[\omega]$ is mapped to integration of $\omega$ over simplicial chains. By Lemma \ref{lem:DS_algebra_isom}, we have $\cehoms{*}(M; \K) \cong \derham^*(M; \K)$. Moreover, by Lemmas \ref{lem:cohom_repr} and \ref{lem:cohom_repr_n}, along with the conformal $\K$-formality of $[g]$, we have $\pharm{g}{*}(M; \K) \cong \cehoms{*}(M)$ by the map $\omega \mapsto [\omega]$. 
	
	It therefore remains to find a suitable embedding $\pharm{g}{*}(M; \K)  \to \wedge^* \K^n$. The idea is to map $\omega \mapsto \omega_x \in \wedge^* T_x^* M \otimes \K \cong \wedge^* \K^n$ for some $x$ satisfying $\rho_g(x) > 0$, and then to use Lemma \ref{lem:abs_value_rigidity} in order to see that this map is injective. The entire remainder of the proof is then merely working around the technicality that our objects are measurable, and hence equalities such as the one given by Lemma \ref{lem:abs_value_rigidity} hold only almost everywhere.
	
	By Corollary \ref{cor:same_support}, all elements $\omega \in \pharm{g}{*}(M; \K)$ share the same support, which we denote $S \subset M$ By Lemma \ref{lem:cohom_repr} and the assumption $H^k(M; \K) \neq \{0\}$, there exists a non-zero element $\omega \in \pharm{g}{*}(M; \K)$ for some $k > 0$. Hence, the set $S$ must necessarily have positive measure.
	
	Note that $\pharm{g}{*}(M; \K)$ is finite dimensional due to being isomorphic to $H^*(M; \R)$, which for closed $M$ is finite-dimensional. We define an inner product $\ip{\cdot}{\cdot}_{\cH}$ on $\pharm{g}{k}(M; \K)$ for every $k \in \{1, \dots, n\}$: for any $\omega_1, \omega_2 \in \pharm{g}{k}(M; \K)$, we let $\ip{\omega_1}{\omega_2}_{\cH}$ be the almost everywhere constant value which the function $\rho^{-2k} \ip{\omega_1}{\omega_2}_g$ assumes on $S$. 
	
	By finite-dimensionality of $\pharm{g}{*}(M; \K)$, we may then select for every $k \in \{1, \dots, n\}$ an orthonormal basis $\{\omega_{k,1}, \dots, \omega_{k,m_k}\}$ of $\pharm{g}{k}(M; \K)$. For $k = 0$, we simply select as the single basis element $\omega_{0,1}$ the constant function $x \mapsto 1$.
	
	We now wish to select a point in $S$ which satisfies specific conditions. We do this by defining sets $S_i \subset S$ corresponding to each desired condition. First, we let $S_1$ be the set of $x \in S$ for which $(\omega_{k,j})_x \neq 0$ for all $k \in \{1, \dots, n\}$ and $j \in \{1, \dots, m_k\}$. For a given pair of indices $(j, k)$, this holds for a.e.\ $x \in S$ since $\spt \omega_{k,j} = S$. Since there are only finitely many such pairs of indices, we have that $S \setminus S_1$ has measure zero.
	
	We then define $S_2$ as the set of points $x \in S$ for which  $\ip{(\omega_{k,i})_x}{(\omega_{k,j})_x}_g = 0$ for all $k \in \{1, \dots, n\}$ and $i,j \in \{1, \dots, m_k\}$. By Lemma \ref{lem:abs_value_rigidity} and the orthonormality of the basis elements under $\ip{\omega_1}{\omega_2}_{\cH}$, this condition holds for almost every $x \in S$ for every such set of indices $(i, j, k)$. Since there are again only finitely many such sets of indices, we have that $S \setminus S_2$ has measure zero.
	
	Suppose then that $k, l \in \{0, \dots, n\}$, $i \in \{1, \dots, m_k\}$ and $j \in \{1, \dots, m_l\}$. Then, since $\pharm{g}{*}(M; \K)$ is closed under $\wedge$, we have
	\begin{equation}\label{eq:S3_eq}
		(\omega_{k,i}) \wedge (\omega_{l,j})
		= a_{i,j,k,l,1} (\omega_{k+l, 1}) + \dots + a_{i,j,k,l,m_{k+l}} (\omega_{k+l, m_{k+l}})
	\end{equation}
	a.e.\ on $M$, for some scalar coefficients $a_{i,j,k,l,1}, \dots, a_{i,j,k,l,m_{k+l}} \in \K$. We let $S_3$ be the set of $x \in S$ at which \eqref{eq:S3_eq} holds for every set of indices $(i, j, k, l)$ as above. Since there are again only finitely many such sets of indices, we have again that $S \setminus S_3$ has zero measure.
	
	Similarly, suppose that $k \in \{1, \dots, n\}$ and $i \in \{1, \dots, m_k\}$. Then we have by Lemma \ref{lem:conformal_hodge} that
	\begin{equation}\label{eq:S4_eq}
		\abs{(\omega_{k,i})}_g^{\frac{n}{k} - 2} \hodge_g \omega_{k,i}
		= b_{i,k,1} \omega_{n-k, 1} + \dots + b_{i,k,m_{n-k}} \omega_{n-k, {m_{n-k}}}
	\end{equation}
	a.e.\ on $M$, for some scalar coefficients $b_{i,k,1}, \dots, b_{i,k,m_{n-k}} \in \K$. We again select $S_4$ to be the set of $x \in S$ at which \eqref{eq:S4_eq} holds for every set of indices $(i, k)$ as above. As previously, we have that $S \setminus S_4$ has zero measure.
	
	Finally, if $[g]$ is not conformally $\K$-formal in the Clifford sense, we select $S_5 = S$. If on the other hand $[g]$ is conformally $\K$-formal in the Clifford sense, then for all $j, k, l \in \{0, \dots, n\}$, $i \in \{1, \dots, m_k\}$ and $i' \in \{1, \dots, m_k\}$, we have
	\begin{equation}\label{eq:S5_eq}
		\dfpart{\omega_{k,i} \odot_g \omega_{l,i'}}_j
		= c_{i, i', j, k, l, 1} \omega_{j, 1} + \dots + c_{i, i', j, k, l, m_j} \omega_{j, m}
	\end{equation}
	a.e.\ on $M$, for some coefficients $c_{i, i', j, k, l, 1}, \dots, c_{i, i', j, k, l, m_j} \in \K$. In this case, we again select $S_5$ to be the set of points of $S$ at which \eqref{eq:S5_eq} holds for every set of indices $(i, i', j, k, l)$ as above, and again obtain that $S \setminus S_5$ is of measure zero.
	
	Now, since $S$ is of positive measure, and since $S \setminus S_i$ is of measure zero for every $i \in \{1, \dots, 5\}$, the set $S' = S_1 \cap S_2 \cap \dots \cap S_5$ is nonempty. Hence, we may select a $x \in S'$.
	
	We then define a function $\kappa \colon \pharm{g}{*}(M; \K) \to (\wedge^* T_x^* M)\otimes \K$ by setting $\kappa(\omega_{k,j}) = (\omega_{k,j})_x$ for all basis elements $\omega_{k,j}$, and then by extending linearly. Since $x \in S_1 \cap S_2$, $\kappa$ maps the basis of every $\pharm{g}{k}(M; \K)$ to a set of nonzero elements in $(\wedge^k T_x^* M) \otimes \K$ which are pairwise orthogonal in $\ip{\cdot}{\cdot}_g$. Therefore, $\kappa$ is injective. Since $x \in S_3$, the map $\kappa$ preserves the wedge product. Hence, $\kappa$ is an injective homomorphism of algebras.
	
	Since $x \in S_4$, the image set $\im \kappa$ has a basis which is mapped inside $\im \kappa$ under the map $\alpha \mapsto \abs{\alpha}_g^{n/k - 2} \hodge_g \alpha$. Since $\im \kappa$ is closed under scalar multiplication, this same basis is also mapped inside $\im \kappa$ by the Hodge star $\hodge_g$. By the linearity of $\hodge_g$, it follows that $\im \kappa$ is closed under $\hodge_g$.
	
	Similarly, suppose that $[g]$ is conformally $\K$-formal in the Clifford sense. Then, since $x \in S_5$, the set $\im \kappa$ has a basis which is mapped inside $\im \kappa$ by the scaled Clifford product $\odot_g$. Thus, since $\im \kappa$ is closed under scalar multiplication, the basis is also mapped inside $\im \kappa$ by the Clifford product $\cdot_g$. Hence, the bilinearity of $\cdot_g$ then implies that $\im \kappa$ is closed under $\cdot_g$.
	
	Finally, we define $\Phi = \Theta \circ \kappa$, where $\Theta \colon (\wedge^* T_x^* M) \otimes \K \to \wedge^* \K^n$ is an orientation-preserving isometric isomorphism of algebras which maps the real part $(\wedge^* T_x^* M) \otimes \R \subset (\wedge^* T_x^* M) \otimes \K$ into $\R^n \subset \K^n$. Since $\Theta$ preserves all relevant properties, the claim follows.
\end{proof}

\section{Consequences of the embedding theorem}\label{sect:consequences}

In this section, we discuss some of the cohomological obstructions for conformal formality implied by Theorems \ref{thm:conf_formal_algebra_embedding} and \ref{thm:conf_formal_clifford_embedding}. In particular, we give a full counterpart to Kotschick's \cite[Theorem 6]{Kotschick_duke} from the geometrically formal theory, as well as explain how the obstruction discussed in Theorem \ref{thm:4_manifold_uqr_obstruction} is obtained from Theorem \ref{thm:conf_formal_clifford_embedding}.

We restrict our discussion in this section to the case $\K = \R$. However, since the complex versions of conformal formality imply the corresponding real ones, every result therefore also immediately applies to the case $\K = \C$. Throughout this section, we denote the standard basis of $\R^n$ by $e_1, \dots, e_n$, and use the shorthand $e_{i_1 i_2 \ldots i_l} = e_{i_1} \wedge e_{i_2} \wedge \dots \wedge e_{i_l}$ for $i_1, \dots, i_l \in \{1, \dots, n\}$.

\subsection{Obstructions for conformal formality}

Let $M$ be a closed, connected, and oriented smooth $n$-manifold. Recall that the \emph{Betti numbers} $b_i$ of $M$ are defined by $b_i = \dim H^i(M; \R)$ for $i \in \{0, \dots, n\}$. An immediate corollary of Theorem \ref{thm:conf_formal_algebra_embedding} is a version of the cohomological dimension bound discussed in the main theorems of \cite{Prywes_Annals} and \cite{Kangasniemi_CohomBound}, as well as in \cite[Theorem 6]{Kotschick_duke}.

\begin{cor}
	Let $M$ be a closed, connected, and oriented smooth $n$-mani\-fold. Suppose that $M$ is conformally formal. Then for every $k \in \{0, \dots, n\}$, the $k$:th Betti number $b_k$ of $M$ satisfies 
	\[
		b_k \leq \binom{n}{k}.
	\]
\end{cor}
\begin{proof}
	Using the graded embedding $\Phi \colon H^*(M; \R) \to \wedge^* \R^n$ of Theorem \ref{thm:conf_formal_algebra_embedding}, we see that $b_k = \dim \Phi(H^k(M; \R)) \leq \dim \wedge^k \R^n = \binom{n}{k}$.
\end{proof}

Similarly, following the argument of another part of \cite[Theorem 6]{Kotschick_duke}, we obtain the following restriction on the first Betti number.

\begin{cor}
	Let $M$ be a closed, connected, and oriented smooth $n$-mani\-fold. Suppose that $M$ is conformally formal. Then the first Betti number $b_1$ of $M$ satisfies 
	\[
		b_1 \neq n-1.
	\]
\end{cor}
\begin{proof}
	Let $\Phi \colon H^*(M; \R) \to \wedge^* \R^n$ be the graded embedding of Theorem \ref{thm:conf_formal_algebra_embedding}. Suppose towards contradiction that $\dim \Phi(H^1(M; \R)) = n-1$. Let $v_1, \dots, v_{n-1}$ be a basis of $\Phi(H^1(M; \R))$. Since $\Phi(H^*(M; \R))$ is closed under $\hodge$ and $\wedge$, we obtain that $v_n = \hodge(v_1 \wedge \dots \wedge v_{n-1}) \in \Phi(H^1(M; \R))$.
	
	However, we have $\ip{v_n}{v_i}e_{12\ldots n} = v_i \wedge \hodge{v_n} = (-1)^{n-1} v_i \wedge (v_1 \wedge \dots \wedge v_{n-1}) = 0$ for every $i \in \{1, \dots, n-1\}$. Since $v_n \in \Span(v_1, \dots, v_{n-1})$ this is only possible if $v_n = 0$. But this is a contradiction, since $v_1 \wedge \dots \wedge v_{n-1} \neq 0$.
\end{proof}

The final part of \cite[Theorem 6]{Kotschick_duke} concerns the splitting of the middle Betti number $b_{2m}$ of a $4m$-manifold. In order to prove our counterpart to it, we briefly recall the definition of this split. Suppose that $\sigma \colon V \times V \to \R$ is a symmetric bilinear form on a finite-dimensional vector space. Then there exists a \emph{$\sigma$-orthogonal} basis $\{v_1, \dots, v_l\}$ of $V$, that is, a basis such that $\sigma(v_i, v_j) = 0$ whenever $i \neq j$. The \emph{signature} of $\sigma$ is the triple $(s_0, s_+, s_-)$, consisting of the numbers of basis elements $v_i$ for which $\sigma(v_i, v_i) = 0$, $\sigma(v_i, v_i) > 0$, and $\sigma(v_i, v_i) < 0$, respectively. Moreover, the signature is independent on the choice of such basis $\{v_1, \dots, v_l\}$.

Suppose then that $M$ is a closed, connected, oriented $4m$-manifold. The \emph{intersection form} of $M$ is a bilinear form $I \colon H^{2m}(M; \R) \times H^{2m}(M; \R) \to \R$ defined by $I(u, u')V = u \cup u'$, where $V$ is the orientation class of $M$. Since $2m$ is even, $I$ is symmetric, and moreover, the first component of the signature of $I$ vanishes due to Poincar\'e duality. Hence, we may write the signature of $I$ as $(0, b_{2m}^+, b_{2m}^-)$, and obtain a split $b_{2m} = b_{2m}^+ + b_{2m}^-$.

\begin{cor}\label{cor:conf_formal_4m_obstruction}
	Let $M$ be a closed, connected, and oriented smooth $4m$-mani\-fold. Suppose that $M$ is conformally formal. Then the two parts $b_{2m}^+, b_{2m}^-$ of the middle Betti number of $M$ satisfy
	\[
		b_{2m}^\pm \leq \frac{1}{2}\binom{4m}{2m}.
	\]
\end{cor}
\begin{proof}
	Let $\Phi$ again be the graded embedding of Theorem \ref{thm:conf_formal_algebra_embedding}. We denote $\Phi(H^{2m}(M; \R)) = \Lambda_{2m}$. We have a symmetric bilinear form $\sigma$ on $\Lambda_{2m}$ defined by $v \wedge v' = \sigma(v, v') e_{12\ldots n}$. Since $\Phi \colon H^{2m}(M; \R) \to \Lambda_{2m}$ is a bijective isomorphism of algebras, it maps $I$-orthogonal bases to $\sigma$-orthogonal bases. It follows that the signature of $\sigma$ is $(0, b_{2m}^{+}, b_{2m}^{-})$ or $(0, b_{2m}^{-}, b_{2m}^{+})$, where the order of the signs is determined by whether $\Phi$ maps the orientation class $V$ of $M$ into a positive or negative multiple of $e_{12\dots n}$.
	
	We then note that, since $\Lambda_{2m}$ is closed under the Hodge star $\hodge$, it has a decomposition $\Lambda_{2m} = \Lambda_{2m}^+ \oplus \Lambda_{2m}^-$ into positive and negative eigenspaces of $\hodge$. Moreover, $\Lambda_{2m}^+$ and $\Lambda_{2m}^-$ are orthogonal to each other, due to the eigenspaces of $\hodge$ in $\wedge^{2m} \R^{4m}$ being orthogonal to each other. Hence, we may select an orthonormal basis $\{v_1, \dots, v_l\}$ of $\Lambda_{2m}$ consisting of elements of $\Lambda_{2m}^+$ and $\Lambda_{2m}^-$. If $i, j \in \{1, \dots, l\}$ and $i \neq j$, then $v_i \wedge v_j = v_i \wedge (\pm \hodge v_j) = \pm \ip{v_i}{ v_j} e_{12\ldots n} = 0$, and therefore $\sigma(v_i, v_j) = 0$. Moreover, $\sigma(v_i, v_i) = 1$ if $v_i \in \Lambda_{2m}^+$, and $\sigma(v_i, v_i) = -1$ if $v_i \in \Lambda_{2m}^-$. Consequently, the signature of $\sigma$ is $(0, \dim \Lambda_{2m}^+, \dim \Lambda_{2m}^-)$.
	
	We conclude that $b_{2m}^{+}, b_{2m}^{-} \in \{\dim \Lambda_{2m}^+, \dim \Lambda_{2m}^-\}$. Since the positive and negative eigenspaces of $\hodge$ in $\wedge^{2m} \R^{4m}$ both are of dimension $\binom{4m}{2m}/2$, we obtain that $\dim \Lambda_{2m}^\pm \leq \binom{4m}{2m}/2$. The claim follows.
\end{proof}

We note that the implications of Theorem \ref{thm:conf_formal_algebra_embedding} are not limited to the above corollaries. This is demonstrated by e.g.\ the resulting obstruction to the conformal formality of $\connsum^{15} (\S^2 \times \S^4)$ discussed in the introduction.

\subsection{The Clifford case}

We then discuss the case of manifolds which are conformally formal in the Clifford sense, in which case Theorem \ref{thm:conf_formal_clifford_embedding} yields extra rigidity for the embedding $\Phi$. We limit the scope of our discussion to explaining the obstruction for 4-manifolds discussed in Theorem \ref{thm:4_manifold_uqr_obstruction}.

\begin{lemma}\label{lem:clifford_4_manifold_lemma}
	Let $M$ be a closed, connected, and oriented smooth $4$-mani\-fold. Suppose that $M$ is conformally formal in the Clifford sense. Then the two parts $b_{2}^+, b_{2}^-$ of the middle Betti number of $M$ satisfy
	\[
		b_{2}^\pm \in \{0, 1, 3\}.
	\]
\end{lemma}

The result follows essentially from the structure of the Clifford product on the positive and negative eigenspaces of $\hodge$. We split the proof in two parts, stating first the part which is not specific to dimension 4.

\begin{lemma}\label{lem:betti_result_product}
	Let $M$ be a closed, connected, and oriented smooth $4m$-mani\-fold, and let $\Phi \colon H^*(M; \R) \to \wedge^* \R^{4m}$ be a graded embedding of algebras such that $\Phi(H^*(M; \R))$ is closed under the Euclidean Clifford product $\cdot$ of $\wedge^* \R^{4m}$. Let $\Lambda_{2m} = \Phi(H^{2m}(M; \R))$, and let $\Lambda_{2m} = \Lambda_{2m}^+ \oplus \Lambda_{2m}^-$ be the decomposition of $\Lambda_{2m}$ into positive and negative eigenspaces of $\hodge$. Moreover, let $P \colon (\wedge^{2m} \R^{4m}) \times (\wedge^{2m} \R^{4m}) \to \wedge^{2m} \R^{4m}$ be the map defined by
	\[
		P(v, v') = \dfpart{v \cdot v'}_{2m}.
	\]
	Then $P(\Lambda_{2m}^+ \times \Lambda_{2m}^+) \subset \Lambda_{2m}^+$ and $P(\Lambda_{2m}^- \times \Lambda_{2m}^-) \subset \Lambda_{2m}^-$.
\end{lemma}
\begin{proof}
	Let $n = 4m$. We first point out that $\Lambda_{2m}$ is indeed closed under $\hodge$, since the Hodge star on $\wedge^{2m} \R^{4m}$ is given by left Clifford multiplication with $(-1)^m e_{12\ldots n}$, and since $e_{12\ldots n} \in \Phi(H^*(M; \R))$. Hence, the decomposition $\Lambda_{2m} = \Lambda_{2m}^+ \oplus \Lambda_{2m}^-$ is indeed valid.
	
	Let $v, v' \in \wedge^{2m} \R^{4m}$. We then have $\hodge \dfpart{v \cdot v'}_{2m} = \dfpart{\hodge (v \cdot v')}_{2m}$. Hence, we obtain by the associativity of the Clifford product that
	\begin{align*}
	\hodge P(v, v') &= \dfpart{(-1)^m e_{12\ldots n} \cdot (v \cdot v')}_{2m}\\
	&\qquad= \dfpart{((-1)^m e_{12\ldots n} \cdot v) \cdot v'}_{2m}
	= P(\hodge v, v').
	\end{align*} 
	By the bilinearity of $P$, it follows that $P$ maps the positive and negative eigenspaces of $\hodge$ to itself. Moreover, since $\Phi(H^*(M; \R))$ is closed under the Clifford product, we have $P(\Lambda_{2m} \times \Lambda_{2m}) \subset \Lambda_{2m}$. The claim follows.
\end{proof}

Lemma \ref{lem:clifford_4_manifold_lemma} now follows from an analysis of the map $P$ in dimension 4.

\begin{proof}[Proof of Lemma \ref{lem:clifford_4_manifold_lemma}]
	By Corollary \ref{cor:conf_formal_4m_obstruction}, we already have $b_2^\pm \leq 3$. It therefore remains to show that $b_2^\pm \neq 2$. Let $\Phi$ be the embedding provided by Theorem \ref{thm:conf_formal_clifford_embedding}, let $\Lambda_2 = \Phi(H^2(M; \R))$, and let $\Lambda_{2} = \Lambda_{2}^+ \oplus \Lambda_{2}^-$ be the split of $\Lambda_{2}$ into positive and negative eigenspaces under $\hodge$.
	
	Consider the map $P \colon (\wedge^2 \R^4) \times (\wedge^2 \R^4) \to \wedge^2 \R^4$ defined in Lemma \ref{lem:betti_result_product}. The positive eigenspace of $\hodge$ in $\wedge^2 \R^4$ has the basis
	\begin{align*}
		f_1 &= \frac{e_{12} + e_{34}}{2},&
		f_2 &= \frac{e_{14} + e_{23}}{2},&
		f_3 &= \frac{e_{13} + e_{42}}{2}.
	\end{align*}
	Using \eqref{eq:clifford_basis_power} and \eqref{eq:clifford_basis_anticommute}, we obtain the following multiplication table for this basis under $P$:
	\begin{align*}
		P(f_f, f_1) &= 0,&
		P(f_1, f_2) &= f_3,&
		P(f_1, f_3) &= -f_2,\\
		P(f_2, f_1) &= -f_3,&
		P(f_2, f_2) &= 0,&
		P(f_2, f_3) &= f_1,\\
		P(f_3, f_1) &= f_2,&
		P(f_3, f_2) &= -f_1,&
		P(f_3, f_3) &= 0.
	\end{align*}
	Hence, $P$ acts on the positive eigenspace of $\hodge$ isomorphically to the cross product in $\R^3$. Since $\Lambda_{2}^+$ is closed under $P$ by Lemma \ref{lem:betti_result_product}, and since no two-dimensional vector subspace of $\R^3$ is closed under the cross product, we therefore conclude that $\dim \Lambda_{2}^+ \neq 2$.
	
	The same proof holds for $\Lambda_{2}^-$. Indeed, the negative eigenspace of $\hodge$ has the basis
	\begin{align*}
		f_1' &= \frac{e_{12} - e_{34}}{2},&
		f_2' &= \frac{e_{13} - e_{42}}{2},&
		f_3' &= \frac{e_{14} - e_{23}}{2},
	\end{align*}
	which has the same multiplication table under $P$. Hence, we conclude that $\dim \Lambda_{2}^\pm \neq 2$. Since $b_2^\pm \in \{\dim \Lambda_{2}^+, \dim \Lambda_{2}^-\}$ as discussed in the proof of Corollary \ref{cor:conf_formal_4m_obstruction}, the claim follows.
\end{proof}

\section{Quasiregular maps}\label{sect:qr_and_structs}

In this section, we briefly go over the essentials of quasiregular maps in preparation for the proofs of Theorems \ref{thm:uqr_elliptic_is_formal} and \ref{thm:uqr_elliptic_is_clifford} in the next section. We focus on how quasiregular maps interact with conformal structures and conformal cohomology theories. For further details on quasiregular maps, we refer the interested reader to e.g.\ the books of Rickman \cite{Rickman_book} and Iwaniec and Martin \cite{Iwaniec-Martin_book}.

Suppose that $M$ and $N$ are closed, connected, oriented Riemannian $n$-manifolds. As stated in the introduction, quasiregular maps are continuous maps in the Sobolev space $W^{1,n}_\loc(M, N)$ which satisfy for some $K \geq 1$ the condition
\begin{equation}\label{eq:qr_condition}
	\abs{Df(x)}^n \leq K J_f(x)
\end{equation}
for almost every $x \in M$. There are several equivalent ways of formally defining the space $W^{1,n}_\loc(M, N)$. Since $f$ is continuous by assumption, a relatively simple approach is to not define the full space $W^{1,n}_\loc(M, N)$, but instead the subspace of continuous functions of $W^{1,n}_\loc(M, N)$, which can be reduced to the Euclidean counterpart using smooth bilipschitz charts.

\subsection{Quasiregular maps and conformal structures} If $f \colon M \to N$ is a non-constant quasiregular map, and  $[g]$ is a bounded conformal structure on $N$, then there exists a \emph{pullback structure} $f^*[g]$, which is a bounded conformal structure on $M$. In particular, $f^*[g] = [f^* g]$, where $f^* g$ is the pullback Riemannian metric defined by
\[
	\ip{v}{w}_{f^* g} = \ip{Df(x) v}{Df(x) w}_{g}
\]
for $v, w \in T_x M$ and a.e.\ $x \in M$. Note that in obtaining a well defined measurable pullback metric $f^* g$, we use two classical properties of quasiregular maps. First, we use the fact that $Df(x)$ is invertible for a.e.\ $x \in M$ in order to see that $f^*g$ gives a positive norm to every nonzero tangent vector at almost every point of $M$. Second, the measurability of the metric $f^* g$ is based on the fact that $f$ satisfies the Lusin $(N^{-1})$-condition, and therefore preserves measurable sets under preimages.

Since we have $f^* (\rho^2 g) = (\rho^2 \circ f) f^* g$ for measurable functions $\rho \colon M \to (0, \infty)$, we obtain a well-defined map of conformal structures. The boundedness of $f^* [g]$ moreover follows from the boundedness of $[g]$ and the quasiregularity of $f$. Indeed, condition \eqref{eq:qr_condition} implies that, for all $v, w \in T_x M \setminus \{0\}$ and a.e.\ $x \in M$, we have $\abs{Df(x)v}_{g_N}/\abs{Df(x) w}_{g_N} \leq K^n \abs{v}_{g_M}/\abs{w}_{g_M}$ in the chosen smooth Riemannian metrics $g_M$ and $g_N$ of $M$ and $N$.

For a.e. $x \in M$, the above formula for the pullback metric takes the form
\[
	\smallip{\alpha \circ (\wedge^k Df(x))}{\beta \circ (\wedge^k Df(x))}_{f^* g} = \ip{\alpha}{\beta}_g
\]
for $k$-covectors $\alpha, \beta \in T_{f(x)}^* N$. In particular, we obtain for measurable differential $k$-forms $\alpha, \beta$ on $N$ the formula
\[
	\ip{f^* \alpha}{f^* \beta}_{f^* g} = \ip{\alpha}{\beta}_g \circ f,
\]
which holds a.e. on $M$. Volume forms follow the formula
\[
	f^* \vol_g = \vol_{f^* g}.
\]
Similarly, the Hodge star of the pullback metric satisfies
\[
	f^* \circ \hodge_g = \hodge_{f^* g} \circ f^*
\]
a.e.\ on $M$.

We then consider the interaction of quasiregular maps with the Clifford product.

\begin{lemma}\label{lem:QR_clifford}
	Let $f \colon M \to N$ be a non-constant non-injective quasiregular map between closed, connected, oriented Riemannian $n$-manifolds, let $[g]$ be a bounded conformal structure on $N$, and let $\K \in \{\R, \C\}$. Suppose that $\alpha \in \Gamma(\wedge^l N; \K)$ and $\beta \in \Gamma(\wedge^m N; \K)$, where $l, m \in \{0, \dots, n\}$. Then
	\[
		(f^* \alpha) \cdot_{f^* g} (f^* \beta) = f^*(\alpha \cdot_g \beta)
	\]
	almost everywhere on $M$.
\end{lemma}
\begin{proof}
	At almost every point $x \in M$, if $\{\eps_1, \dots, \eps_n\}$ is an $\ip{\cdot}{\cdot}_g$-orthonormal basis of self-conjugate elements of $T^*_{f(x)}N \otimes \K$, then $\{f^*\eps_1, \dots, f^*\eps_n\}$ is an $\ip{\cdot}{\cdot}_{f^* g}$-orthonormal basis of self-conjugate elements of $T^*_{x}M \otimes \K$. Suppose that $B = \{\eps_1, \dots, \eps_n\}$ is such a basis at $f(x) \in N$. 
	
	Consider two induced basis elements $\eps_I = \eps_{i_1} \cdot_g \dots \cdot_g \eps_{i_l} = \eps_{i_1} \wedge \dots \wedge \eps_{i_l}$ and $\eps_{I'} = \eps_{i_1'} \wedge \dots \wedge \eps_{i_m'}$. By \eqref{eq:clifford_basis_power} and \eqref{eq:clifford_basis_anticommute}, their product is of the form $\eps_I \cdot_g \eps_{I'} = (-1)^{c(I, I')} \eps_{j_1} \wedge \dots \wedge \eps_{j_k}$. Now we may compute, using again \eqref{eq:clifford_basis_power} and \eqref{eq:clifford_basis_anticommute}, that
	\begin{align*}
		(f^*\eps_I) \cdot_{f^*g} (f^*\eps_{I'})
		&= (f^*(\eps_{i_1} \wedge \dots \wedge \eps_{i_l})) \cdot_{f^* g}
			(f^*(\eps_{i_1'} \wedge \dots \wedge \eps_{i_m'}))\\
		&= ((f^*\eps_{i_1}) \wedge \dots \wedge (f^*\eps_{i_l})) \cdot_{f^* g}
			((f^*\eps_{i_1'}) \wedge \dots \wedge (f^*\eps_{i_m'}))\\
		&= (-1)^{c(I, I')} (f^*\eps_{j_1}) \wedge \dots \wedge (f^*\eps_{j_k})\\
		&= f^*(\eps_I \cdot_g \eps_J).
	\end{align*}
	Hence, by writing $\alpha_{f(x)}$ and $\beta_{f(x)}$ with the basis induced by $B$, the claim follows from the bilinearity of the Clifford product, linearity of $f^*$, and the interaction of $f^*$ with the scalar multiplication.
\end{proof}

Combining Lemma \ref{lem:QR_clifford} with the formula for $\ip{f^* \alpha}{f^* \beta}_{f^* g}$ immediately yields the following corollary.

\begin{cor}\label{cor:QR_clifford_scaled}
	Let $f \colon M \to N$ be a non-constant non-injective quasiregular map between closed, connected, oriented Riemannian $n$-manifolds, let $[g]$ be a bounded conformal structure on $N$, and let $\K \in \{\R, \C\}$. Suppose that $\alpha \in \Gamma(\wedge^l N; \K)$ and $\beta \in \Gamma(\wedge^m N; \K)$, where $l, m \in \{0, \dots, n\}$. Then
	\[
	(f^* \alpha) \odot_{f^* g} (f^* \beta) = f^*(\alpha \odot_g \beta)
	\]
	almost everywhere on $M$.
\end{cor}

\subsection{Quasiregular maps and conformal cohomology}

A major motivation for the conformal cohomology theories of Section \ref{sect:conf_cohomology} is that the pullback map $f^*$ of differential forms for a quasiregular $f \colon M \to N$ naturally induces a pullback map in cohomology. In particular, we will use the following result.

\begin{lemma}\label{lem:qr_pullback_in_cohom}
	Let $M, N$ be closed, connected, oriented Riemannian $n$-mani\-folds, let $f^* \colon M \to N$ be a non-constant quasiregular map, and let $\K \in \{\R, \C\}$. Then we obtain a chain map
	\[
		f^* \colon \cesobs(\wedge^* N; \K) \to \cesobs(\wedge^* M; \K).
	\]
	In particular, $f$ naturally induces a map
	\[
		f^* \colon \cehoms{*}(N; \K) \to \cehoms{*}(M; \K).
	\]
\end{lemma}

While we are not aware of a full proof of the above in an existing reference, the essential ideas already exist in the known literature. Indeed, Donaldson and Sullivan \cite{Donaldson-Sullivan_Acta} show this for $n = 4$ and $f$ bijective, and the above case is not essentially different. A version for $\cehom{*}(M; \R)$ is also shown by the author and Pankka in \cite[Lemma 3.4]{Kangasniemi-Pankka_PLMS}. See also the related discussion by Gol'dshtein and Troyanov in \cite[Theorem 6.6]{Goldshtein-Troyanov_DeRham}.

However, for the convenience of the reader, we collect a detailed proof here.

\begin{proof}[Proof of Lemma \ref{lem:qr_pullback_in_cohom}]
	The map
	\[
		f^* \colon \cesob(\wedge^* N; \K) \to \cesob(\wedge^* M; \K).
	\]
	is a chain map; see \cite[Lemma 3.4]{Kangasniemi-Pankka_PLMS} for the case $\K = \R$, where the case $\K = \C$ is an immediate consequence. Since $\cesobs(\wedge^* N; \K) \subset \cesob(\wedge^* N; \K)$, the only remaining part is then to show that $f^*\cesobs(\wedge^* N; \K) \subset \cesobs(\wedge^* M; \K)$. Since $f$ is continuous, clearly $f^* C(\wedge^0 N; \K) \subset C(\wedge^0 M; \K)$. The remaining goal is therefore to show that
	\begin{equation}
		f^* L^{\frac{n}{k}, \sharp}(\wedge^k N; \K) \subset L^{\frac{n}{k}, \sharp}(\wedge^k M; \K)
	\end{equation}
	for $k \in \{1, \dots, n\}$.
	
	The method of the proof is as in \cite[Lemma 3.1 and Corollary 3.2]{Kangasniemi-Pankka_PLMS}, where the case $k = n$ is shown. Let $g_M$ and $g_N$ be the Riemannian metrics of $M$ and $N$, respectively. Recall that there exists $r > 1$ for which $J_f \in L^r(M; \R)$; see Elcrat and Meyers \cite{Meyers-Elcrat_HigherInt} or Martio \cite{Martio_HigherInt}. Suppose that $\omega \in L^p(\wedge^k M; \K)$, where $p > n/k$. 
	
	We let
	\[
		q = \left(\frac{r}{r + \frac{k}{n}p - 1}\right) p.
	\]
	Since $p > n/k$, we get that $q < p$, and therefore $p/q > 1$. Moreover, we have
	\[
		\frac{kq}{n} = \frac{krp}{n(r-1) + kp} = \frac{r}{1 + \frac{n}{kp}(r-1)}
		> \frac{r}{1 + (r - 1)} = 1.
	\]
	Hence, $q > n/k$.
	
	By \cite[(2.4)]{Kangasniemi-Pankka_PLMS}, we have
	\[
		\abs{f^* \omega}_{g_M} \leq C (\abs{\omega}_{g_N} \circ f) J_f^\frac{k}{n}
	\]
	a.e.\ on $M$. Hence, we estimate
	\begin{align*}
		&\int_M \abs{f^*\omega}_{g_M}^q \vol_{g_M}
		\leq C^q \int_M (\abs{\omega}_{g_N}^q \circ f) J_f^\frac{kq}{n} \vol_{g_M}\\
		&\hspace{2cm}\leq C^q \int_M (\abs{\omega}_{g_N}^q \circ f) J_f^\frac{q}{p}
			J_f^{q(\frac{k}{n} - \frac{1}{p})} \vol_{g_M}\\
		&\hspace{2cm}\leq C^q \left( \int_M (\abs{\omega}_{g_N}^p \circ f) 
			J_f \vol_{g_M}\right)^\frac{q}{p} \left(
			\int_M J_f^{\frac{pq}{p-q}(\frac{k}{n} - \frac{1}{p})}
			\vol_{g_M}\right)^\frac{p-q}{p}\\
		&\hspace{4cm}= C^q (\deg f)^\frac{q}{p} \norm{\omega}_{p, g_N}^q \left(
			\int_M J_f^{\frac{pq}{p-q}(\frac{k}{n} - \frac{1}{p})}
			\vol_{g_M}\right)^\frac{p-q}{p}.
	\end{align*}
	The claim then follows, since a simple calculation using the definition of $q$ yields
	\[
		\frac{pq}{p-q}\left( \frac{k}{n} - \frac{1}{p} \right)
		= \frac{q}{p-q}\left( \frac{k}{n}p - 1 \right)
		= \frac{q}{p-q}\left( \frac{rp}{q} - r \right)
		= r.
	\]
\end{proof}

We then note that, if $[g]$ is a bounded conformal structure on $N$ and $f \colon M \to N$ is non-constant quasiregular, then the cohomology representation of $\cehoms{*}(N; \K)$ by $\pharm{g}{*}(N; \K)$ is in fact mapped by $f^*$ into the cohomology representation of $\cehoms{*}(M; \K)$ by $\pharm{f^*g}{*}(M; \K)$.

\begin{lemma}\label{lem:p_harmonic_pullback}
	Let $M, N$ be closed, connected, oriented Riemannian $n$-mani\-folds, let $f \colon M \to N$ be a non-constant quasiregular map, let $[g]$ be a bounded conformal structure on $N$, let $k \in \{0, \dots, n-1\}$, and let $\K \in \{\R, \C\}$. Then 
	\[
		f^* \pharm{g}{k}(N; \K) \subset \pharm{f^*g}{k}(M; \K).
	\]
\end{lemma}
\begin{proof}
	The case $k = 0$ is trivial, as $\pharm{g}{k}(N; \K)$ and $\pharm{f^*g}{k}(M; \K)$ consist of constant functions. Let then $\omega \in \pharm{g}{k}(N; \K)$ for $k \in \{1, \dots, n-1\}$. By Lemmas \ref{lem:cohom_repr} and \ref{lem:qr_pullback_in_cohom}, we obtain that $f^* \omega \in L^{n/k}(\wedge^k M; \K)$ and $d f^*\omega = f^* d\omega = 0$. For the remaining condition, we compute that
	\begin{multline*}
		d \left(\abs{f^* \omega}_{f^* g}^{\frac{n}{k} - 2} \hodge_{f^* g} f^*\omega \right)
		=d \left( \left(\abs{\omega}_{g}^{\frac{n}{k} - 2} \circ f \right)
			f^* \hodge_{g} \omega \right)\\
		= d f^* \left( \abs{\omega}_{g}^{\frac{n}{k} - 2} 
			\hodge_{g} \omega\right)
		= f^* d \left( \abs{\omega}_{g}^{\frac{n}{k} - 2} 
			\hodge_{g} \omega\right)
		= 0.
	\end{multline*}
\end{proof}

Finally, we end this section with the observation that the conformal Hodge decomposition defined in Section \ref{sect:Hodge_theory} is preserved under quasiregular maps. More precisely, we have the following.

\begin{lemma}\label{lem:conf_hodge_decomposition_qr}
	Let $M, N$ be closed, connected, oriented Riemannian $n$-mani\-folds, let $f \colon M \to N$ be a non-constant quasiregular map, and let $[g]$ be a bounded conformal structure on $N$. Let $\omega \in L^{n/k}(\wedge^k M; \K)$, where $k \in \{1, \dots, n-1\}$ and $\K \in \{\R, \C\}$. Let $(d\alpha, d\beta, \gamma)$ be the conformal Hodge decomposition of $\omega$ with respect to $[g]$. Then $(f^* d\alpha, f^* d\beta, f^* \gamma)$ is the conformal Hodge decomposition of $f^*\omega$ with respect to $[f^* g]$.
\end{lemma}
\begin{proof}
	By the uniqueness of conformal Hodge decompositions provided by Proposition \ref{prop:nonlinear_Hodge_decomposition}, it suffices to show that $(f^* d\alpha, f^* d\beta, f^* \gamma)$ satisfies the conditions of a conformal Hodge decomposition of $f^* \omega$ with respect to $[f^*g]$.
	
	By using the Sobolev-Poincar\'e inequality, we may assume that $\alpha, \beta \in \cesob(\wedge^* M)$. In particular, we therefore have by \cite[Lemma 3.4]{Kangasniemi-Pankka_PLMS} that $df^*\alpha = f^*d\alpha \in L^{n/k}(\wedge^k M; \K)$ and $df^*\beta = f^*d\beta \in L^{n/(n-k)}(\wedge^{n-k} M; \K)$. Moreover, by Lemma \ref{lem:p_harmonic_pullback}, we have $f^* \gamma \in \pharm{f^* g}{k}(M; \K)$.
	
	Hence, the forms $f^* d\alpha$, $f^* d\beta$ and $f^* \gamma$ are in the correct spaces, and it remains to check that
	\[
		f^* \omega = f^* d\alpha + \abs{f^* d\beta}_{f^* g}^{\frac{n}{n-k} - 2} \hodge_{f^*g} f^* d\beta + f^* \gamma.
	\]
	This follows immediately from the computation
	\begin{multline*}
		\abs{f^*d\beta}_{f^*g}^{\frac{n}{n-k} - 2} 
			\hodge_{f^* g} f^*d\beta
		= \left( \abs{d\beta}_{g}^{\frac{n}{n-k} - 2} \circ f \right)
			f^*(\hodge_g d\beta)\\
		= f^* \left( \abs{d\beta}_{g}^{\frac{n}{n-k} - 2} 
			\hodge_g d\beta \right).
	\end{multline*}
\end{proof}

\section{Proof of Theorems \ref{thm:uqr_elliptic_is_formal} and \ref{thm:uqr_elliptic_is_clifford}}\label{sect:formality_of_uqr_ell}

Let now $M$ be a closed, connected, oriented Riemannian $n$-manifold, and let $f \colon M \to M$ be a non-constant, non-injective uniformly quasiregular self-map on $M$. We denote by $g_0$ the smooth Riemannian metric of $M$.

In this case, the map $f$ has an invariant conformal structure $[g_f]$. In particular, $[g_f]$ is a bounded conformal structure on $M$ for which $f^*[g_f] = [g_f]$. This was shown by Iwaniec and Martin in \cite{Iwaniec-Martin_AASF}, generalizing a previous proof by Tukia \cite{Tukia_QCGroups} for quasiconformal groups. The proof was given for $M = \S^n$, but the method works for all closed, connected and oriented $M$. We also assume that $g_f \in [g_f]$ is the element of the structure for which $\vol_{g_f} = \vol_{g_0}$. In this case, we have $f^* g_f = J_f^{2/n} g_f$.

In this section, we show that the structure $[g_f]$ is conformally formal in the Clifford sense. Hence, uniformly quasiregularly elliptic manifolds are conformally formal in the Clifford sense.

\subsection{Eigenvector-based spaces.}

Let $k \in \{1, \dots, n-1\}$. Since $f^*[g_f] = [g_f]$, we therefore obtain by Lemma \ref{lem:p_harmonic_pullback} and the conformal invariance of the space $\pharm{g_f}{k}(M; \K)$ a self-map $f^* \colon \pharm{g_f}{k}(M; \K) \to \pharm{g_f}{k}(M; \K)$. We use the abbreviation $\pharm{f}{k}(M; \K) = \pharm{g_f}{k}(M; \K)$.

For every $c \in \cehoms{k}(M; \K)$, we use $\omega_c$ to denote the unique element of $\pharm{f}{k}(M; \K) \cap c$ provided by Lemma \ref{lem:cohom_repr}. As discussed in \cite[Section 4.2]{Kangasniemi_CohomBound}, we have $f^*\omega_c \in \cehoms{k}(M; \K) \cap f^*c$, and therefore by the uniqueness of the representatives we have $f^* \omega_c = \omega_{f^* c}$. Moreover, since $\pharm{f}{k}(M; \K)$ is preserved under scalar multiplication, we also have $\omega_{\lambda c} = \lambda \omega_c$ for every $\lambda \in \K$. In particular, if $c \in \cehoms{k}(M; \K)$ is an eigenvector of $f^*$ satisfying $f^* c = \lambda c$, then also $f^* \omega_c = \lambda \omega_c$.

Our strategy in proving that $\pharm{f}{k}(M; \K)$ is linear is to consider a different space $\cE^k_f(M; \C)$, which is constructed from linear combinations of eigenvectors $\omega_c$ of $f^*$. Here, we require complex coefficients, as the definition of $\cE^k_f(M; \C)$ requires a basis of eigenvectors of $f^*$. The space $\cE^k_f(M; \C)$ shares properties with $\pharm{f}{k}(M; \C)$ such as unique representation of cohomology and being mapped to itself by $f^*$. However, the space $\cE^k_f(M; \C)$ is linear by definition. The crux of the proof is then that we show that $\pharm{f}{k}(M; \C) = \cE^k_f(M; \C)$ by a suitable weak convergence argument.

We now give a detailed definition of the space $\cE^k_f(M; \C)$.

\begin{defn}
	Let $f \colon M \to M$ be a non-constant non-injective uniformly quasiregular map on a closed, connected, oriented Riemannian $n$-manifold $M$, where $n \geq 2$. Moreover, let $k \in \{1, \dots, n-1\}$. 
		
	Fix a basis $\cB = \{c_1, \dots, c_l\}$ of $\cehoms{k}(M; \C)$ consisting of eigenvector classes of $f^* \colon \cehoms{k}(M; \C) \to \cehoms{k}(M; \C)$, with corresponding eigenvalues $\lambda_1, \dots, \lambda_l$. We then denote $\omega_j = \omega_{c_j}$ for every $j \in \{1, \dots, l\}$. Now, the space $\cE^k_f(M; \C)$ of differential $k$-forms is defined by
	\[
		\cE^k_f(M; \C) = \Span (\omega_{1}, \dots, \omega_{l}).
	\]
\end{defn}

Note that the spaces $\cE^k_f(M; \C)$ are a-priori dependent on the selection of basis $\cB$, although later on in the proof it turns out that this is not the case. We now record the basic properties of the spaces $\cE^k_f(M; \C)$.

\begin{lemma}\label{lem:eigenvector_space_properties}
	Let $f \colon M \to M$ be a non-constant non-injective uniformly quasiregular map on a closed, connected, oriented Riemannian $n$-manifold $M$, where $n \geq 2$. Let $k \in \{1, \dots, n-1\}$. Then for every $c \in \cehoms{k}(M; \C)$, there exists a unique $\eta_c \in c \cap \cE^*_f(M; \C)$. The elements $\eta_c$ satisfy
	\begin{gather}
		\label{eq:prop_mult} z\eta_c = \eta_{zc},\\
		\label{eq:prop_add} \eta_c + \eta_{c'} = \eta_{c + c'}, \text{ and}\\
		\label{eq:prop_pull} f^* \eta_c = \eta_{f^* c}
	\end{gather}
	for all $c, c' \in \cehoms{k}(M; \C)$, $k \in \{0, \dots, n\}$ and $z \in \C$.
\end{lemma}

\begin{proof}
	Since $f^* \colon \cehoms{k}(M; \C) \to \cehoms{k}(M; \C)$ is diagonalizable by the main result of \cite{Kangasniemi-Pankka_PLMS}, it follows that its eigenspaces form a direct sum decomposition of $\cehoms{k}(M; \C)$. The existence of a unique $\eta_c \in c \cap \cE^k_f(M; \C)$ for every $c \in \cehoms{k}(M; \C)$ follows immediately from this and the definition of $\cE^k_f(M; \C)$.
	
	We then verify the properties \eqref{eq:prop_mult}, \eqref{eq:prop_add}, and \eqref{eq:prop_pull}. Let $c, c' \in \cehoms{k}(M; \C)$ for some $k \in \{0, \dots, n\}$, and write $c = a_1 c_1 + \dots + a_l c_l$ and $c' = a_1' c_1 + \dots a_l' c_l$, where $\{c_i\}$ is the eigenvector basis of $\cehoms{k}(M; \C)$ used in the definition of $\cE_f^k(M; \C)$. Then $zc + c' = (za_1 + a_1') c_1 + \dots + (za_l + a_l') c_l$, and it follows that
	\[
		\eta_{zc + c'} = (za_1 + a_1') \omega_{c_1} + \dots + (za_l + a_l') \omega_{c_l} = z \eta_{c} + \eta_{c'}.
	\]
	Hence, \eqref{eq:prop_mult} and \eqref{eq:prop_add} follow. Similarly, we have $f^*c = f^*(a_1 c_1 + \dots + a_l c_l) = a_1 \lambda_1 c_1 + \dots + a_l \lambda_l c_l$, and therefore
	\[
		\eta_{f^* c} = a_1 \lambda_1 \omega_{c_1} + \dots + a_l \lambda_l \omega_{c_l}
		= a_1 f^*(\omega_{c_1}) + \dots + a_l f^*(\omega_{c_l}) = f^* \eta_c.
	\]
	Hence, \eqref{eq:prop_pull} also follows.
\end{proof}

\subsection{Proof of $\cE^k_f(M; \C) = \pharm{f}{k}(M; \C)$}

We begin by recording a lemma which is used multiple times in the proof of the main results. It is essentially a version of \cite[Corollary 5.2]{Kangasniemi_CohomBound} for more general differential forms, which in turn is based on an argument used by Okuyama and Pankka in \cite[Theorem 5.2]{Okuyama-Pankka_measure}. The proof is practically the same as in \cite{Kangasniemi_CohomBound}, but we provide it for completeness.

\begin{lemma}\label{lem:weak_convergence_to_zero}
	Let $f \colon M \to M$ be a non-constant non-injective uniformly quasiregular map on a closed, connected, oriented Riemannian $n$-manifold $M$, where $n \geq 2$. Let $k \in \{1, \ldots, n\}$, let $\K \in \{\R, \C\}$, let $\lambda \in \K$ be such that $\abs{\lambda} = (\deg f)^{k/n}$, and let $\omega \in \cesobs(\wedge^k M; \K) \cap \ker(d)$ be such that $[\omega] = [0]$ in $\cehoms{k}(M; \K)$. Then
	\[
		\frac{(f^j)^* \omega}{\lambda^j} \to 0,
	\]
	where the convergence is in the weak sense.
\end{lemma}
\begin{proof}
	Since $[\omega] = [0]$, there exists $\tau \in \cesobs(\wedge^{k-1} M; \K)$ for which $\omega = d\tau$. Since $\cesobs(\wedge^{k-1} M; \K) \subset L^{n/(k-1)}(\wedge^{k-1} M; \K)$, we therefore have $\tau \in L^{n/(k-1)}(\wedge^{k-1} M; \K)$. Note that for $k = 1$, we interpret $n/(k-1) = \infty$; the above still holds since $\cesobs(\wedge^0 M; \K) \subset C(\wedge^0 M; \K) \subset L^\infty(\wedge^0 M; \K)$.
	
	Now let $\varphi \in C^\infty(\wedge^{n-k} M; \K)$ be a smooth test form. We obtain
	\begin{align*}
		\int_M \varphi \wedge \lambda^{-j} (f^j)^* \omega 
		&= \lambda^{-j}  \int_M \varphi \wedge (f^j)^* d\tau\\
		&= \lambda^{-j}  \int_M d\varphi \wedge (f^j)^* \tau.
	\end{align*}
	By the H\"older-type inequality for wedge products, we have
	\begin{align*}
		\abs{\int_M \varphi \wedge \lambda^{-j} (f^j)^* \omega} 
		&\leq C(n) (\deg f)^{-\frac{jk}{n}} \norm{d\varphi}_{\frac{n}{n-k+1}} 
		\norm{(f^j)^* \tau}_{\frac{n}{k-1}}.
	\end{align*}
	Finally, if $k > 1$, we have by the estimate from \cite[Lemma 2.4]{Kangasniemi-Pankka_PLMS}, that $\smallnorm{(f^j)^* \tau}_{n/(k-1)} \leq C(n, K) (\deg f)^{(k-1)/n} \norm{\tau}_{n/(k-1)}$. If $k = 1$, then this also holds, since $\tau$ is an $L^\infty$-function, and therefore $\smallnorm{(f^j)^*\tau}_\infty = \smallnorm{\tau \circ f^j}_\infty = \norm{\tau}_\infty$. Hence,
	\begin{align*}
		\abs{\int_M \varphi \wedge \lambda^{-j} (f^j)^* \omega} 
		&\leq C(n, K) (\deg f)^{\frac{j(k - 1)}{n}-\frac{jk}{n}} 
		\norm{d\varphi}_{\frac{n}{n-k+1}} \norm{\tau}_{\frac{n}{k-1}}\\
		&=  C(n, K) (\deg f)^{-\frac{j}{n}} 
		\norm{d\varphi}_{\frac{n}{n-k+1}} \norm{\tau}_{\frac{n}{k-1}}.
	\end{align*}
	Since the right hand side converges to zero, the claim follows.
\end{proof}

Now, we are ready to show the linearity of $\pharm{f}{k}(M; \K)$.

\begin{lemma}\label{lem:closed_under_addition}
	Let $f \colon M \to M$ be a non-constant non-injective uniformly quasiregular map on a closed, connected, oriented Riemannian $n$-manifold $M$, where $n \geq 2$, and let $k \in \{1, \dots, n-1\}$. Then 
	\[
		\pharm{f}{k}(M; \C) = \cE^k_f(M; \C).
	\]
	In particular, $\pharm{f}{k}(M; \C)$ is closed under addition, as is consequently also $\pharm{f}{k}(M; \R)$.
\end{lemma}
\begin{proof}
	We wish to prove that in fact $\eta_c = \omega_c$ for every $c \in \cehoms{k}(M; \C)$. Hence, let $c \in \cehoms{k}(M; \C)$ for $k \in \{1, \dots, n-1\}$. We define a map $F$ of complex measurable differential forms by
	\[
		F \omega = (\deg f)^{-\frac{k}{n}} f^* \omega.
	\]
	Since $f^*[g_f] = [g_f]$, we have 
	\begin{equation}\label{eq:F_preserves_norm}
		\norm{F \omega}_{n/k, g_f} = \norm{\omega}_{n/k, g_f}
	\end{equation}
	for all $\omega \in L^{n/k}(\wedge^k M; \C)$; see \cite[Section 4.2]{Kangasniemi_CohomBound}. By Lemma \ref{lem:p_harmonic_pullback}, $F$ maps $\pharm{f}{k}(M; \C)$ into itself. Moreover, by Lemma \ref{lem:eigenvector_space_properties}, $F$ maps $\cE^k_f(M; \C)$ into itself.
	
	We now consider the sequence
	\begin{equation}\label{eq:pullback_sequence}
		\left( F^j (\eta_c - \omega_c) \right)_{j=1}^\infty.
	\end{equation}
	Since $[\eta_c] = [\omega_c]$, and therefore $[\eta_c - \omega_c] = [0]$ in $\cehoms{k}(M; \C)$, we obtain by Lemma \ref{lem:weak_convergence_to_zero} that \eqref{eq:pullback_sequence} converges to 0 weakly.
	
	The space $\cE^k_f(M; \C)$ equipped with $\norm{\cdot}_{n/k, g_f}$ is a finite dimensional normed space. Hence, it is bilipischitz equivalent to a Euclidean space. By \eqref{eq:F_preserves_norm}, $(F^j \eta_c)_{j=1}^\infty$ is a bounded sequence in $\cE^k_f(M; \C)$. Since bounded subsets of Euclidean spaces are precompact, and since $\norm{\cdot}_{n/k, g_f}$ is equivalent with the standard $L^{n/k}$-norm by \eqref{eq:norm_comparison}, there exists a subsequence $(F^{j_i} \eta_c)_{i = 1}^\infty$  which converges to some measurable form $\eta$ in the space $L^{n/k}(\wedge^k M; \C)$.
	
	We then apply Theorem \ref{thm:ISS_continuity_complex} and \eqref{eq:F_preserves_norm}, and obtain that
	\begin{multline*}
		\norm{F^{j_{i_1}} \omega_c - F^{j_{i_2}} \omega_c}_{\frac{n}{k}, g_f}\\
		\leq C \left( \norm{F^{j_{i_1}} \omega_c}_{\frac{n}{k}, g_f} + 
			\norm{F^{j_{i_2}} \omega_c}_{\frac{n}{k}, g_f} \right)^{1-t}
			\norm{F^{j_{i_1}} \eta_c - F^{j_{i_2}} \eta_c}_{\frac{n}{k}, g_f}^t\\
		= C (2 \norm{\omega_c}_{\frac{n}{k}, g_f})^{1-t} 
			\norm{F^{j_{i_1}} \eta_c - F^{j_{i_2}} \eta_c}_{\frac{n}{k}, g_f}^t.
	\end{multline*}
	Thus, we conclude that the sequence $(F^{j_i} \omega_c)_{i=1}^\infty$ is also Cauchy, and hence also converges to some $\omega$ in the space $L^{n/k}(\wedge^k M; \C)$.
	
	We therefore have $F^j (\eta_c - \omega_c) \to 0$ weakly, and also $F^{j_i} (\eta_c - \omega_c) \to \eta - \omega$ in $L^{n/k}(\wedge^k M; \C)$. Hence, by uniqueness of the limit, we conclude that $F^{j_i} \left(\eta_c - \omega_c\right) \to 0$ in $L^{n/k}(\wedge^k M; \C)$. It follows that $\norm{F^{j_i} (\eta_c - \omega_c)}_{n/k, g_f} \to 0$. But by \eqref{eq:F_preserves_norm}, $\norm{F^{j} (\eta_c - \omega_c)}_{n/k, g_f} = \norm{\eta_c - \omega_c}_{n/k, g_f}$ for every $j$. Hence, we conclude that $\norm{\eta_c - \omega_c}_{n/k, g_f} = 0$, and therefore $\omega_c = \eta_c$. The claim follows.
\end{proof}

\subsection{Algebra structure}

So far we have achieved that $\pharm{f}{k}(M; \K)$ is a vector space for every $k \in \{0, \dots, n-1\}$. It remains therefore to select a suitable $\pharm{f}{n}(M; \K)$, and to show that the resulting $\pharm{f}{*}(M; \K)$ is closed under $\odot_{g_f}$.

Our choice of $\pharm{f}{n}(M; \K)$ is as follows. Let $\mu_f$ denote the $f$-invariant measure of Okuyama and Pankka; see \cite{Okuyama-Pankka_measure}. We select the space $\pharm{f}{n}(M; \K)$ to consist of all $\K$-multiples of $\mu_f$ which are represented by an integrable $n$-form. Since $\mu_f$ is a measure, we have $\pharm{f}{n}(M; \K) \subset \pharmgeq{g_f}{n}(M; \K)$. Moreover, since $f^* \mu_f = (\deg f) \mu_f$ by \cite[Theorem 2]{Okuyama-Pankka_measure}, the space $\pharm{f}{n}(M; \K)$ is also preserved under the pullback $f^*$.

We point out that if $M$ is not a rational cohomology sphere, then $\mu_f$ is represented by an $n$-form by \cite[Theorem 1.2]{Kangasniemi_CohomBound}. Hence, in this case, we have $\dim_\K \pharm{f}{n}(M; \K) = 1$. If $M$ is a rational cohomology sphere, then either $\mu_f$ has an $n$-form representation and $\dim_\K \pharm{f}{n}(M; \K) = 1$, or $\mu_f$ doesn't have one and $\pharm{f}{n}(M; \K) = \{0\}$.

We now show that $\pharm{f}{*}(M; \K)$ is closed under $\odot_{g_f}$, and therefore $[g_f]$ is conformally $\K$-formal in the Clifford sense. The crux of the proof lies in the special case where two eigenvectors of $f^*$ are multiplied by $\odot_{g_f}$. For this, we start with a lemma where we use the conformal Hodge decomposition of Section \ref{sect:Hodge_theory} to show that essentially all eigenvectors of $f^*$ are in $\pharm{f}{*}(M; \K)$.

\begin{lemma}\label{lem:eigenvectors_are_harmonic}
	Let $f \colon M \to M$ be a non-constant non-injective uniformly quasiregular map on a closed, connected, oriented Riemannian $n$-manifold $M$, where $n \geq 2$. Let $k \in \{1, \dots, n\}$, and let $\K \in \{\R, \C\}$. Suppose that $\omega \in L^{n/k, \sharp}(\wedge^k M; \K)$ satisfies $f^* \omega = \lambda \omega$ for some $\lambda \in \K$. Then $\omega \in \pharm{f}{k}(M; \K)$.
\end{lemma}
\begin{proof}
	We may suppose $\omega \neq 0$, since $0 \in \pharm{f}{k}(M; \K)$. We therefore have $\abs{\lambda} = (\deg f)^{k/n}$; see e.g.\ the computation in \cite[(6.2)]{Kangasniemi_CohomBound}.
	
	Consider first the case $k < n$. By Proposition \ref{prop:nonlinear_Hodge_decomposition}, $\omega$ has a unique conformal Hodge decomposition $(d\alpha, d\beta, \gamma)$ with respect to $[g_f]$. Moreover, since $\omega \in L^{n/k, \sharp}(\wedge^k M; \K)$, Proposition \ref{prop:nonlinear_Hodge_decomposition} also implies that $d\alpha \in L^{n/k, \sharp}(\wedge^k M; \K)$ and $d\beta \in L^{n/(n-k), \sharp}(\wedge^{n-k} M; \K)$. By the Sobolev-Poincar\'e inequality, or in the case of $1$-forms by the Sobolev embedding theorem, we may therefore assume that $\alpha \in \cesobs(\wedge^{k-1} M; \K)$ and $\beta \in \cesobs(\wedge^{n-k-1} M; \K)$.
	
	By Lemma \ref{lem:conf_hodge_decomposition_qr}, we have that the conformal Hodge decomposition of $f^*\omega$ with respect to $[f^* g_f]$ is $(f^*d\alpha, f^*d\beta, f^*\gamma)$. Since $[g_f]$ is $f$-invariant, we in fact have $[f^* g_f] = [g_f]$. Moreover, we also see that the conformal Hodge decomposition of $\lambda \omega$ with respect to $[g_f]$ is $(\lambda d\alpha, \abs{\lambda}^{(n-k)/k-1}\overline{\lambda} d\beta, \lambda \gamma)$. Since $f^* \omega = \lambda \omega$ and the decompositions are unique, we therefore must have
	\begin{align*}
		f^* d\alpha &= \lambda d\alpha,\\
		f^* d\beta &= \left(\abs{\lambda}^{\frac{n-k}{k} - 1}
			 \overline{\lambda}\right) d\beta,
		\text{ and}\\
		f^* \gamma &= \lambda \gamma.
	\end{align*}
	
	It follows that $d\alpha$ is a fixed point of $\lambda^{-1} f^*$. However, since $[d\alpha] = [0]$ in $\cehoms{k}(M; \K)$, we may now use Lemma \ref{lem:weak_convergence_to_zero} to conclude that $(\lambda^{-1} f^*)^j d\alpha \to 0$ weakly. Hence, $d\alpha = 0$. The same argument also yields $d\beta = 0$. Hence, we have obtained that $\omega = \gamma \in \pharm{f}{k}(M; \K)$, which concludes the case $k < n$.
	
	Finally, we treat the case $k = n$. In this case, we have $\omega \in L^{1, \sharp}(\wedge^n M; \K) = \cesobs(\wedge^n M; \K) = \cesobs(\wedge^n M; \K) \cap \ker d$, and therefore $\omega$ belongs into a cohomology class $[\omega]$ of $\cehoms{n}(M; \K)$. If $[\omega] = [0]$, then $(\lambda^{-1} f^*)^j \omega \to 0$ weakly by Lemma \ref{lem:weak_convergence_to_zero}, and therefore $\omega = 0 \in \pharm{f}{n}(M; \K)$. If on the other hand $[\omega] \neq [0]$, then $\int_M \omega \neq 0$. By the change of variables formula for quasiregular maps, we obtain that
	\[
		\lambda \int_M \omega = \int_M f^* \omega = (\deg f) \int_M \omega.
	\]
	Hence, we have $\lambda = \deg f$. Since $((\deg f)^{-1} f^*)^j \omega$ converges weakly to a multiple of $\mu_f$ by \cite[Theorem 5.1]{Kangasniemi_CohomBound}, we conclude that $\omega$ represents a multiple of $\mu_f$, and therefore $\omega \in \pharm{f}{n}(M; \K)$.
\end{proof}

We then apply Lemma \ref{lem:eigenvectors_are_harmonic} to the $\odot_{g_f}$-product of two eigenvectors.

\begin{lemma}\label{lem:closed_under_wedges}
	Let $f \colon M \to M$ be a non-constant non-injective uniformly quasiregular map on a closed, connected, oriented Riemannian $n$-manifold $M$, where $n \geq 2$. Let $l, m \in \{0, \dots, n\}$, and let $\K \in \{\R, \C\}$. Suppose that $\omega_1 \in \pharm{f}{l}(M; \K)$ and $\omega_2 \in \pharm{f}{m}(M; \K)$ satisfy $f^* \omega_1 = \lambda_1 \omega_1$ and $f^* \omega_2 = \lambda_2 \omega_2$ for some $\lambda_1, \lambda_2 \in \K$. Then
	\[
		\omega_1 \odot_{g_f} \omega_2 \in \pharm{f}{*}(M; \K).
	\]
\end{lemma}
\begin{proof}
	Let $k \in \{0, \dots, n\}$, and denote $\omega = \dfpart{\omega_1 \odot_{g_f} \omega_2}_k$. We wish to show that $\omega \in \pharm{f}{k}(M; \K)$. We may assume $k > 0$ and $(l, m) \neq (0, 0)$, as the remaining cases are trivial; see the second paragraph of the proof of Lemma \ref{lem:clifford_nonlinear_basis_lemma}.
	
	Now, we may compute using Corollary \ref{cor:QR_clifford_scaled} that
	\begin{multline*}
		f^* \omega
		= \dfpart{f^*(\omega_1 \odot_{g_f} \omega_2)}_k
		= \dfpart{(f^*\omega_1) \odot_{f^* g_f} (f^*\omega_2)}_k\\
		= \big\langle(\lambda_1\omega_1) \odot_{J_f^{2/n} g_f} (\lambda_2\omega_2)\big\rangle_k
		= \dfpart{(\lambda_1\omega_1) \odot_{g_f} (\lambda_2\omega_2)}_k
		= \frac{\lambda_1 \lambda_2}{\abs{\lambda_1\lambda_2}^{
			\frac{l+m-k}{l+m}}} \omega.
	\end{multline*}
	Hence, we have $f^*\omega = \lambda \omega$ with $\lambda = \abs{\lambda_1 \lambda_2}^{k/(l+m) - 1} \lambda_1 \lambda_2$. Moreover, Lemma \ref{lem:Clifford_holder} yields the estimate
	\[
		\abs{\omega}_{g_f} 
		= \abs{\dfpart{\omega_1 \cdot_{g_f} \omega_2}_k}_{g_f}^\frac{k}{l+m} 
		\leq C(n) \left( \abs{\omega_1}_{g_f} \abs{\omega_2}_{g_f}
			\right)^\frac{k}{l+m}
	\]
	Since $\omega_1 \in L^{n/l, \sharp}(\wedge^l M; \K)$ and $\omega_2 \in L^{n/m, \sharp}(\wedge^m M; \K)$, it follows by Hölder's inequality that $\omega \in L^{n/k, \sharp}(\wedge^m M; \K)$. Hence, $\omega \in \pharm{f}{k}(M; \K)$ by Lemma \ref{lem:eigenvectors_are_harmonic}, concluding the proof.
\end{proof}

We then compile the results so far to reach our desired conclusion.

\begin{thm}\label{thm:uqr_elliptic_is_formal_C}
	Let $M$ be a closed, connected, oriented Riemannian $n$-mani\-fold for $n \geq 2$. Suppose that $M$ admits a non-constant, non-injective uniformly quasiregular self-map $f \colon M \to M$. Then $M$ and the invariant conformal structure $[g_f]$ are conformally $\K$-formal, and also conformally $\K$-formal in the Clifford sense, for both $\K = \R$ and $\K = \C$.
\end{thm}

\begin{proof}
	It suffices to show the result for $\K = \C$, since the complex versions of addition, scalar product, $\wedge$ and $\odot_{g_f}$ on $\pharm{f}{*}(M; \C)$ restrict to $\pharm{f}{*}(M; \R)$ as the corresponding real versions. Moreover, it suffices to show conformal $\C$-formality in the Clifford sense, as it implies conformal $\C$-formality. By Lemma \ref{lem:closed_under_addition}, the space $\pharm{f}{*}(M; \C)$ is linear. Moreover, $\pharm{f}{*}(M; \C)$ has a graded basis consisting of eigenvectors of $f^*$. By Lemma \ref{lem:closed_under_wedges}, this eigenvector basis is mapped into $\pharm{f}{*}(M; \C)$ by $\odot_{g_f}$. Hence, $\pharm{f}{*}(M; \C)$ is closed under $\odot_{g_f}$ by Lemma \ref{lem:clifford_nonlinear_basis_lemma}. The claim follows.
\end{proof}

With the proof of Theorem \ref{thm:uqr_elliptic_is_formal_C} complete, Theorems \ref{thm:uqr_elliptic_is_formal} and \ref{thm:uqr_elliptic_is_clifford} immediately follow. Theorem \ref{thm:4_manifold_uqr_obstruction} then follows from a combination of Theorem \ref{thm:uqr_elliptic_is_clifford} and Lemma \ref{lem:clifford_4_manifold_lemma}. Finally, Theorem \ref{thm:uqr_ellipticity_is_stronger} is an immediate consequence of Theorem \ref{thm:4_manifold_uqr_obstruction}.


\bibliographystyle{abbrv}
\bibliography{sources}

\end{document}